\pdfoutput=1
\documentclass[12pt,article]{amsart}
\usepackage{xcolor}
\usepackage[all,color]{xy}
\usepackage{amsmath,amssymb,amsthm,amscd,amsxtra,amsfonts,mathrsfs,graphicx,bm,slashed,array,mathtools,setspace,amsfonts,mathtools,enumitem}
\input xy
\xyoption{all}
\newtheorem{Theorem}{Theorem}[section]
\newtheorem{Lemma}[Theorem]{Lemma} 
\newtheorem{Proposition}[Theorem]{Proposition}

\newtheorem{Example}[Theorem]{Example}
\newtheorem{Remark}[Theorem]{Remark}
\newtheorem{Question}[Theorem]{Question}
\newtheorem{Definition}[Theorem]{Definition}

\newtheorem{Notation}[Theorem]{Notation}
\newtheorem*{Theorem A}{Theorem A}
\newtheorem*{Theorem C}{Theorem C}

\newcommand*{\overbar}[1]{\mkern 1.5mu\overline{\mkern-1.5mu#1\mkern-1.5mu}\mkern 1.5mu}
\advance\evensidemargin-.5in
\advance\oddsidemargin-.5in
\advance\textwidth1in

\begin{document}
\author{Charlie Beil}
 \address{Institut f\"ur Mathematik und Wissenschaftliches Rechnen, Universit\"at Graz, Heinrichstrasse 36, 8010 Graz, Austria.}
 \email{charles.beil@uni-graz.at}
\title[Nonnoetherian homotopy dimer algebras and NCCRs]{Nonnoetherian homotopy dimer algebras and noncommutative crepant resolutions}
 \date{}
 
\begin{abstract}
Noetherian dimer algebras form a prominent class of examples of noncommutative crepant resolutions (NCCRs). 
However, dimer algebras which are noetherian are quite rare, and we consider the question: how close are nonnoetherian homotopy dimer algebras to being NCCRs? 
To address this question, we introduce a generalization of NCCRs to nonnoetherian tiled matrix rings. 
We show that if a noetherian dimer algebra is obtained from a nonnoetherian homotopy dimer algebra $A$ by contracting each arrow whose head has indegree 1, then $A$ is a noncommutative desingularization of its nonnoetherian center.
Furthermore, if any two arrows whose tails have indegree 1 are coprime, then $A$ is a nonnoetherian NCCR. 
\end{abstract}

\maketitle

\section{Introduction}

Let $(R,\mathfrak{m})$ be a local domain with an algebraically closed residue field $k$.
In the mid 1950's, Auslander, Buchsbaum, and Serre established the famous homological characterization of regularity in the case $R$ is noetherian \cite{AB,AB2,S}: 
$R$ is regular if and only if
$$\operatorname{gldim}R = \operatorname{pd}_R(k) = \operatorname{dim}R.$$

In 1984, Brown and Hajarnavis generalized this characterization to the setting of noncommutative noetherian rings which are module-finite over their centers \cite{BH}: such a ring $A$ with local center $R$ is said to be homologically homogeneous if for each simple $A$-module $V$,
\begin{equation*} \label{homhom3}
\operatorname{gldim}A = \operatorname{pd}_A(V) = \operatorname{dim}R.
\end{equation*}

In 2002, Van den Bergh placed this notion in the context of derived categories with the introduction of noncommutative crepant resolutions (henceforth NCCRs).
Specifically, a homologically homogeneous ring $A$ is a (local) NCCR if $R$ is a normal Gorenstein domain and $A$ is the endomorphism ring of a finitely generated reflexive $R$-module \cite[Definition 4.1]{V}.\footnote{A proper birational map $f: Y \to X$ from a non-singular variety $Y$ to a Gorenstein singularity $X$ is a crepant resolution if $f^* \omega_X = \omega_Y$. 
Given an NCCR $A$ of $R = k[X]$, Van den Bergh conjectured that the bounded derived category of $A$-modules is equivalent to the bounded derived category of coherent sheaves on $Y$ \cite[Conjecture 4.6]{V}.}

A prominent class of NCCRs are noetherian dimer algebras on a torus (Definition \ref{dimer def}) \cite{Br,Bo,D,B4,B7}. 
In fact, every 3-dimensional affine toric Gorenstein singularity admits an NCCR given by such a dimer algebra \cite{G,IU}.
Although dimer quivers may be defined on any compact surface, in this article we consider the case where the surface is a torus. 

A \textit{homotopy algebra} is the quotient of a dimer algebra by homotopy-like relations on the paths in its quiver; a dimer algebra coincides with its homotopy algebra if and only if it is noetherian \cite[Theorem 1.1]{B4}. 
Homotopy algebras, just like noetherian dimer algebras, are tiled matrix rings over polynomial rings.
The homotopy algebra of a nonnoetherian dimer algebra is also nonnoetherian and an infinitely generated module over its nonnoetherian center. 
Here we consider the question:
\begin{center}
\textit{How close are nonnoetherian homotopy algebras to being NCCRs?} 
\end{center}

To address this question, we consider a relatively small but important class of nonnoetherian homotopy algebras: Let $A$ be a homotopy algebra with quiver $Q$ such that a noetherian dimer algebra is obtained by contracting each arrow of $Q$ whose head has indegree 1, and no arrow of $Q$ has head and tail of indegree both 1. 
Denote by $R$ the center of $A$.
The scheme $\operatorname{Spec}R$ has a unique closed point $\mathfrak{m}_0$ of positive geometric dimension \cite[Theorem 1.1]{B6}.
Furthermore, $\mathfrak{m}_0$ is the unique closed point for which the localizations 
$$R_{\mathfrak{m}_0} \ \ \ \text{ and } \ \ \ A_{\mathfrak{m}_0} := A \otimes_R R_{\mathfrak{m}_0}$$ 
are nonnoetherian \cite[Section 3]{B6}, \cite[Theorem 3.4]{B3}.
An initial answer to our question appears to be negative: 
\begin{itemize}
 \item $A_{\mathfrak{m}_0}$ has infinite global dimension (Proposition \ref{infinite global dimension}).
 \item $A_{\mathfrak{m}_0}$ is typically not the endomorphism ring of a module over its center.
\end{itemize}

However, the underlying structure of $A_{\mathfrak{m}_0}$ is more subtle.
To uncover this structure, we introduce a generalization of homological homogeneity and NCCRs for nonnoetherian tiled matrix rings.
Let $A$ be a nonnoetherian tiled matrix ring with local center $(R,\mathfrak{m})$.
Firstly, we introduce
\begin{itemize}
 \item the \textit{cycle algebra} $S$ of $A$, which is a commutative algebra that contains the center $R$ as a subalgebra (but in general is not a subalgebra of $A$); and
 \item the \textit{cyclic localization} $A_{\mathfrak{q}}$ of $A$ at a prime ideal $\mathfrak{q}$ of $S$.
\end{itemize}

We then say $A$ is \textit{cycle regular} if for each $\mathfrak{q} \in \operatorname{Spec}S$ minimal over $\mathfrak{m}$ and each simple $A_{\mathfrak{q}}$-module $V$, we have
$$\operatorname{gldim}A_{\mathfrak{q}} = \operatorname{pd}_{A_{\mathfrak{q}}}(V) = \operatorname{dim}S_{\mathfrak{q}}.$$ 
Furthermore, we say $A$ is a \textit{nonnoetherian NCCR} if the cycle algebra $S$ is a noetherian normal Gorenstein domain, $A$ is cycle regular, and for each $\mathfrak{q} \in \operatorname{Spec}S$ minimal over $\mathfrak{m}$, $A_{\mathfrak{q}}$ is the endomorphism ring of a reflexive module over its center $Z(A_{\mathfrak{q}})$. 

Our main result is the following. 
\newpage

\begin{Theorem} \label{yesterday} (Theorems \ref{height}, \ref{big theorem 1}, \ref{NCCR theorem}.)
Let $A$ be a nonnoetherian homotopy algebra such that a noetherian dimer algebra is obtained by contracting each arrow whose head has indegree 1, and no arrow of $A$ has head and tail of indegree both 1. 
Then
\begin{enumerate} 
 \item $A_{\mathfrak{m}_0}$ is cycle regular.
  \item For each prime $\mathfrak{q}$ of the cycle algebra $S$ which is minimal over $\mathfrak{m}_0$, we have
\begin{equation*} \label{ilovezephy}
\operatorname{gldim} A_{\mathfrak{q}} = \dim S_{\mathfrak{q}} = \operatorname{ght}_R(\mathfrak{m}_0) = 1 < 3 = \operatorname{ht}_R(\mathfrak{m}_0) = \operatorname{dim}R_{\mathfrak{m}_0},
\end{equation*}
where $\operatorname{ght}_R(\mathfrak{m}_0)$ and $\operatorname{ht}_R(\mathfrak{m}_0)$ denote the geometric height and height of $\mathfrak{m}_0$ in $R$ respectively.
Furthermore, for each prime $\mathfrak{q}$ of $S$ minimal over $\mathfrak{q} \cap R$,
$$\operatorname{gldim} A_{\mathfrak{q}} = \operatorname{ght}_R(\mathfrak{q}\cap R).$$ 
 \item If the arrows whose tails have indegree 1 are pairwise coprime, then $A_{\mathfrak{m}_0}$ is a nonnoetherian NCCR.
\end{enumerate}
\end{Theorem}

The second claim suggests that geometric height, rather than height, is the `right' notion of codimension for nonnoetherian commutative rings, noting that geometric height and height coincide for noetherian rings \cite[Theorem 3.8]{B5}. 
An example of a dimer algebra which is a nonnoetherian NCCR is given in Figure \ref{second figure}, and described in Example \ref{second example}.

This work is a continuation of \cite{B3}, where the author considered localizations $A_{\mathfrak{p}} := A \otimes_R R_{\mathfrak{p}}$ of nonnoetherian dimer and homotopy algebras $A$  at points $\mathfrak{p} \in \operatorname{Spec}R$ away from $\mathfrak{m}_0$. 
We focus exclusively on homotopy algebras here since the localization of a dimer algebra at $\mathfrak{m}_0$ is much less tractable than its homotopy counterpart; for example, any dimer algebra satisfying the assumptions of Theorem \ref{yesterday} has a free subalgebra, whereas its homotopy algebra does not \cite{B4}.

In future work we hope to explore the implications of the definitions we have introduced in terms of derived categories and tilting theory, and to study larger classes of nonnoetherian homotopy algebras, as well as other classes of tiled matrix rings.

\begin{figure}
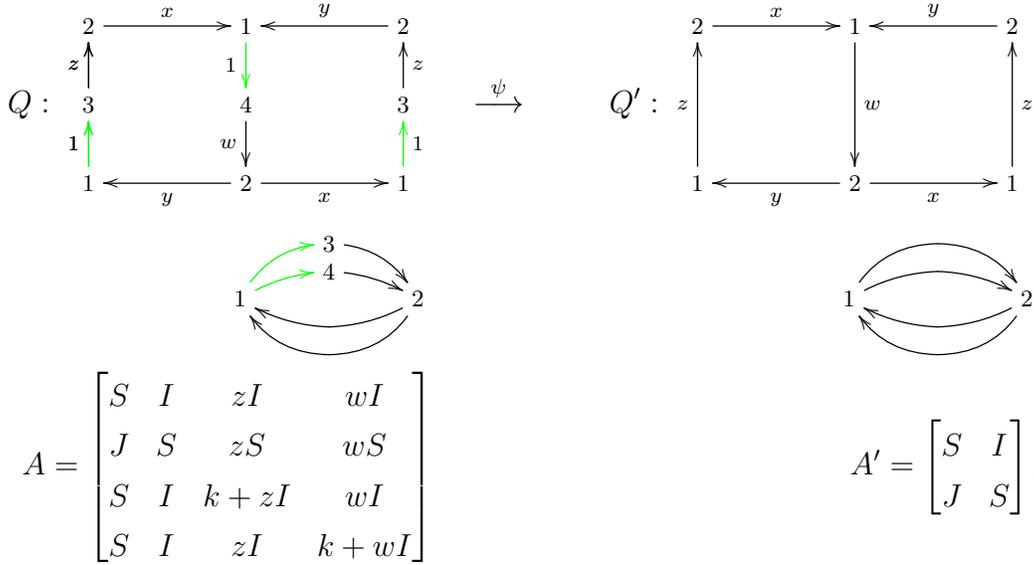

\begin{align*}
Q: \xy 0;/r.31pc/:
(-16,8)*+{\text{\scriptsize{$2$}}}="1";
(0,8)*+{\text{\scriptsize{$1$}}}="2";
(16,8)*+{\text{\scriptsize{$2$}}}="3";
(-16,-8)*+{\text{\scriptsize{$1$}}}="4";
(0,-8)*+{\text{\scriptsize{$2$}}}="5";
(16,-8)*+{\text{\scriptsize{$1$}}}="6";
(-16,0)*+{\text{\scriptsize{$3$}}}="7";
(16,0)*+{\text{\scriptsize{$3$}}}="8";
(0,0)*+{\text{\scriptsize{$4$}}}="9";
{\ar^x"1";"2"};{\ar_y"3";"2"};
{\ar^y"5";"4"};{\ar_x"5";"6"};
{\ar@[green]^1"4";"7"};{\ar^z"7";"1"};
{\ar@[green]^1"4";"7"};{\ar^z"7";"1"};
{\ar@[green]_1"6";"8"};{\ar_z"8";"3"};
{\ar@[green]_1"2";"9"};{\ar_w"9";"5"};
\endxy
& \ \ \ \stackrel{\psi}{\longrightarrow} \ \ \ &
Q': \xy 0;/r.31pc/:
(-16,8)*+{\text{\scriptsize{$2$}}}="1";
(0,8)*+{\text{\scriptsize{$1$}}}="2";
(16,8)*+{\text{\scriptsize{$2$}}}="3";
(-16,-8)*+{\text{\scriptsize{$1$}}}="4";
(0,-8)*+{\text{\scriptsize{$2$}}}="5";
(16,-8)*+{\text{\scriptsize{$1$}}}="6";
{\ar^x"1";"2"};{\ar_y"3";"2"};
{\ar^y"5";"4"};{\ar_x"5";"6"};
{\ar^z"4";"1"};
{\ar_z"6";"3"};
{\ar^w"2";"5"};
\endxy
\\
\ \ \ \ \ \ \ \xy 0;/r.35pc/:
(-8,0)*+{\text{\scriptsize{$1$}}}="1";
(8,0)*+{\text{\scriptsize{$2$}}}="2";
(0,5)*+{\text{\scriptsize{$3$}}}="3";
(0,2.5)*+{\text{\scriptsize{$4$}}}="4";
(0,-2.5)*{}="5";
(0,-5)*{}="6";
{\ar@/^/@[green]"1";"3"};
{\ar@/^/"3";"2"};
{\ar@/^.2pc/@[green]"1";"4"};
{\ar@/^.2pc/"4";"2"};
{\ar@{-}@/^.2pc/"2";"5"};
{\ar@/^.2pc/"5";"1"};
{\ar@{-}@/^/"2";"6"};
{\ar@/^/"6";"1"};
\endxy
&&
\ \ \ \ \ \ \ \xy 0;/r.35pc/:
(-8,0)*+{\text{\scriptsize{$1$}}}="1";
(8,0)*+{\text{\scriptsize{$2$}}}="2";
(0,5)*{}="3";
(0,2.5)*{}="4";
(0,-2.5)*{}="5";
(0,-5)*{}="6";
{\ar@{-}@/^/"1";"3"};
{\ar@/^/"3";"2"};
{\ar@{-}@/^.2pc/"1";"4"};
{\ar@/^.2pc/"4";"2"};
{\ar@{-}@/^.2pc/"2";"5"};
{\ar@/^.2pc/"5";"1"};
{\ar@{-}@/^/"2";"6"};
{\ar@/^/"6";"1"};
\endxy
\\
A = \left[ \begin{matrix}
S & I & zI & wI \\ 
J & S & zS & wS \\
S & I & k + zI & wI \\
S & I & zI & k + wI
\end{matrix} \right]
& & A' = \left[ \begin{matrix} S & I \\ J & S \end{matrix} \right].
\end{align*}
\caption{(Example \ref{second example}.)  The homotopy algebra $A$ is a nonnoetherian NCCR.
The quivers $Q$ and $Q'$ on the top line are each drawn on a torus, and the two contracted arrows of $Q$ are drawn in green.
Here, $S = k[xz,yz,xw,yw]$ is the coordinate ring for the quadric cone, considered as a subalgebra of the polynomial ring $k[x,y,z,w]$, and $I$ and $J$ are the respective $S$-modules $(x,y)S$ and $(z,w)S$.} 
\label{second figure}
\end{figure}

\section{Preliminary definitions} \label{dimer algebras}

Throughout, let $k$ be an algebraically closed field, let $S$ be an integral domain and a $k$-algebra, and let $R$ be a (possibly nonnoetherian) subalgebra of $S$.
Denote by $\operatorname{Max}S$, $\operatorname{Spec}S$, and $\operatorname{dim}S$ the maximal spectrum (or variety), prime spectrum (or affine scheme), and Krull dimension of $S$ respectively; similarly for $R$. 
For a subset $I \subset S$, set $\mathcal{Z}(I) := \left\{ \mathfrak{n} \in \operatorname{Max}S \ | \ \mathfrak{n} \supseteq I \right\}$.

A quiver $Q =(Q_0,Q_1,\operatorname{t},\operatorname{h})$ consists of a vertex set $Q_0$, an arrow set $Q_1$, and head and tail maps $\operatorname{h}, \operatorname{t}: Q_1 \to Q_0$. 
Denote by $\operatorname{deg}^+ i$ the indegree of a vertex $i \in Q_0$; by $kQ$ the path algebra of $Q$; and by $e_i \in kQ$ the idempotent at vertex $i$. 
Path concatenation is read right to left.
By module and global dimension we mean left module and left global dimension, unless stated otherwise.
In a fixed matrix ring, denote by $e_{ij}$ the matrix with a $1$ in the $ij$-th slot and zeros elsewhere, and set $e_i := e_{ii}$. 

The following definitions were introduced in \cite{B5} to formulate a theory of geometry for nonnoetherian rings with finite Krull dimension. 

\begin{Definition} \label{def dep} \rm{\cite[Definition 3.1]{B5}
\begin{itemize}
 \item We say $S$ is a \textit{depiction} of $R$ if $S$ is a finitely generated $k$-algebra, the morphism
$$\iota_{S/R} : \operatorname{Spec}S \rightarrow \operatorname{Spec}R, \ \ \ \ \mathfrak{q} \mapsto \mathfrak{q} \cap R,$$ 
is surjective, and
$$\left\{ \mathfrak{n} \in \operatorname{Max}S \ | \ R_{\mathfrak{n}\cap R} = S_{\mathfrak{n}} \right\} = \left\{ \mathfrak{n} \in \operatorname{Max}S \ | \ R_{\mathfrak{n} \cap R} \text{ is noetherian} \right\} \not = \emptyset.$$
 \item The \textit{geometric height} of $\mathfrak{p} \in \operatorname{Spec}R$ is the minimum
$$\operatorname{ght}(\mathfrak{p}) := \operatorname{min} \left\{ \operatorname{ht}_S(\mathfrak{q}) \ | \ \mathfrak{q} \in \iota^{-1}_{S/R}(\mathfrak{p}), \ S \text{ a depiction of } R \right\}.$$
The \textit{geometric dimension} of $\mathfrak{p}$ is
$$\operatorname{gdim} \mathfrak{p} := \operatorname{dim}R - \operatorname{ght}(\mathfrak{p}).$$
 \end{itemize}
} \end{Definition}

The algebras that we will consider in this article are called homotopy (dimer) algebras.
Dimer algebras are a type of quiver with potential, and were introduced in string theory \cite{BFHMS} (see also \cite{BD}).
Homotopy algebras are special quotients of dimer algebras, and were introduced in \cite{B2}. 

\begin{Definition} \label{dimer def} \rm{ \

$\bullet$ Let $Q$ be a finite quiver whose underlying graph $\overbar{Q}$ embeds into a two-dimensional real torus $T^2$, such that each connected component of $T^2 \setminus \overbar{Q}$ is simply connected and bounded by an oriented cycle, called a \textit{unit cycle}.\footnote{In contexts such as cluster algebras, $\overbar{Q}$ may be embedded into any compact surface; see for example \cite{BKM}.}\textsuperscript{,}\footnote{Note that for any vertex $i \in Q_0$, the indegree and outdegree of $i$ are equal.}\textsuperscript{,}\footnote{In \cite{B1}, it useful to allow length $1$ unit cycles.  Consequently, it is possible for a length $1$ path $a \in Q_1$ to equal a vertex modulo $I$; in this case, $a$ is called a `pseudo-arrow' rather than an `arrow', in order to avoid modifying standard definitions such as perfect matchings.} 
The \textit{dimer algebra} of $Q$ is the quiver algebra $kQ/I$ with relations
$$I := \left\langle p - q \ | \ \exists \ a \in Q_1 \text{ such that } pa \text{ and } qa \text{ are unit cycles} \right\rangle \subset kQ,$$
where $p$ and $q$ are paths.

Since $I$ is generated by certain differences of paths, we may refer to a path modulo $I$ as a \textit{path} in the dimer algebra $kQ/I$.

$\bullet$ Two paths $p,q \in kQ/I$ form a \textit{non-cancellative pair} if $p \not = q$, and there is a path $r \in kQ/I$ such that 
$$rp = rq \not = 0 \ \ \text{ or } \ \ pr = qr \not = 0.$$
$kQ/I$ and $Q$ are called \textit{non-cancellative} if there is a non-cancellative pair; otherwise they are called \textit{cancellative}. 
By \cite[Theorem 1.1]{B4}, $kQ/I$ is noetherian if and only if it is cancellative.

$\bullet$ We call the quotient algebra 
$$A := (kQ/I)/\left\langle p - q \ | \ p,q \text{ is a non-cancellative pair} \right\rangle$$
the \textit{homotopy (dimer) algebra} of $Q$.\footnote{A dimer algebra coincides with its homotopy algebra if and only if its quiver is cancellative.}
(For the definition of a homotopy algebra on a general surface, see \cite{B2}.)

$\bullet$ Let $A$ be a (homotopy) dimer algebra with quiver $Q$.
 \begin{itemize}
  \item[--] A \textit{perfect matching} $D \subset Q_1$ is a set of arrows such that each unit cycle contains precisely one arrow in $D$.
  \item[--] A \textit{simple matching} $D \subset Q_1$ is a perfect matching such that $Q \setminus D$ supports a simple $A$-module of dimension $1^{Q_0}$ (that is, $Q \setminus D$ contains a cycle that passes through each vertex of $Q$). 
 Denote by $\mathcal{S}$ the set of simple matchings of $A$. 
 \end{itemize}
}\end{Definition} 

\section{Cycle algebra and nonnoetherian NCCRs}

In this section we introduce the cycle algebra, cyclic localization, and nonnoetherian NCCRs.
Let $B$ be an integral domain and a $k$-algebra. 
Let
\begin{equation*}
A = \left[ A^{ij} \right] \subset M_d(B)
\end{equation*} 
be a tiled matrix algebra; that is, each diagonal entry $A^i := A^{ii}$ is a unital subalgebra of $B$.
Denote by $Z = Z(A)$ the center of $A$.

\begin{Definition} \rm{
Set
\begin{equation*}
R := k\left[\cap_{i = 1}^d A^i \right] \ \ \ \text{ and } \ \ \ S := k\left[ \cup_{i = 1}^d A^i \right].
\end{equation*}
We call $S$ the \textit{cycle algebra} of $A$.
Furthermore, for $\mathfrak{q} \in \operatorname{Spec}S$, set
\begin{equation*}
A_{\mathfrak{q}} := \left\langle \left[ \begin{matrix} A^{1}_{\mathfrak{q} \cap A^{1}} & A^{12} & \cdots & A^{1d} \\ A^{21} & A^{2}_{\mathfrak{q} \cap A^{2}} && \\
\vdots & & \ddots & \\ A^{d1} & & & A^{d}_{\mathfrak{q} \cap A^{d}} \end{matrix} \right] \right\rangle \subset M_d(\operatorname{Frac}B).
\end{equation*}
We call $A_{\mathfrak{q}}$ the \textit{cyclic localization} of $A$ at $\mathfrak{q}$.
}\end{Definition}

Note that $R$ and $S$ are integral domains since they are subalgebras of $B$.
The following definitions aim to generalize homological homogeneity and NCCRs to the nonnoetherian setting. 

\begin{Definition} \label{NCCR2} \rm{
Suppose $R$ is a local domain with unique maximal ideal $\mathfrak{m}$.
\begin{itemize}
 \item We say $A$ is \textit{cycle regular} if for each $\mathfrak{q} \in \operatorname{Spec}S$ minimal over $\mathfrak{m}$ and each simple $A_{\mathfrak{q}}$-module $V$,
$$\operatorname{gldim}A_{\mathfrak{q}} = \operatorname{pd}_{A_{\mathfrak{q}}}(V) = \operatorname{dim}S_{\mathfrak{q}}.$$
  \item We say $A$ is a \textit{noncommutative desingularization} if $A$ is cycle regular, and $A \otimes_R \operatorname{Frac}R$ and $\operatorname{Frac}R$ are Morita equivalent.
  \item We say $A$ is a \textit{nonnoetherian noncommutative crepant resolution} if $S$ is a normal Gorenstein domain, $A$ is cycle regular, and for each $\mathfrak{q} \in \operatorname{Spec}S$ minimal over $\mathfrak{m}$, $A_{\mathfrak{q}}$ is the endomorphism ring of a reflexive $Z(A_{\mathfrak{q}})$-module.
\end{itemize}
}\end{Definition}

\begin{Remark} \rm{
Suppose $B$ is a finitely generated $k$-algebra, and $k$ is uncountable.
Further suppose the embedding $\tau: A \hookrightarrow M_d(B)$ has the properties that 
\begin{enumerate}[label=(\roman*)]
 \item for generic $\mathfrak{b} \in \operatorname{Max}B$, the composition
\begin{equation*} \label{A to}
A \stackrel{\tau}{\longrightarrow} M_d(B) \stackrel{1}{\longrightarrow} M_d\left(B/\mathfrak{b} \right)
\end{equation*}
is surjective; 
 \item the morphism 
\begin{equation*} \label{B to}
\operatorname{Max}B \rightarrow \operatorname{Max}\tau(Z), \ \ \ \ \mathfrak{b} \mapsto \mathfrak{b}\mathbf{1}_d \cap \tau(Z),
\end{equation*}
is surjective; and
 \item for each $\mathfrak{n} \in \operatorname{Max}S$, $R_{\mathfrak{n} \cap R} = S_{\mathfrak{n}}$ iff $R_{\mathfrak{n} \cap R}$ is noetherian.
\end{enumerate}
$(\tau,B)$ is then said to be an \textit{impression} of $A$ \cite[Definition 2.1]{B7}.

Under these conditions, the center $Z$ of $A$ is equal to $R$,
$$Z = R \mathbf{1}_d,$$
and is depicted by $S$ \cite[Theorem 4.1.1]{B5}. 
Furthermore, by \cite[Theorem 4.1.2]{B5}, 
\begin{align*} \label{R = S}
R = S \ \ \Leftrightarrow & \ \ \text{$A$ is a finitely generated $R$-module}\\
\Leftrightarrow & \ \ \text{$R$ is noetherian} \nonumber \\
\Rightarrow & \ \ \text{$A$ is noetherian} \nonumber
\end{align*}
In particular, if $R$ is noetherian, then the cyclic and central localizations of $A$ at $\mathfrak{q} \in \operatorname{Spec}S$ are isomorphic algebras,
$$A_{\mathfrak{q}} \cong A \otimes_R R_{\mathfrak{q} \cap R}.$$

If $\mathfrak{p} \in \operatorname{Spec}R$ and $\mathfrak{q} \in \operatorname{Spec}S$, then we denote by $A_{\mathfrak{p}}$ and $A_{\mathfrak{q}}$ the central and cyclic localizations of $A$ respectively; no ambiguity arises since the two localizations coincide whenever $R = S$.
}\end{Remark}

\section{A class of nonnoetherian homotopy algebras} \label{section 4}

For the remainder of this article, we will consider a class of homotopy algebras whose quivers contain vertices with indegree 1.
Such quivers are necessarily non-cancellative.
Unless stated otherwise, let $A$ be a nonnoetherian homotopy algebra with quiver $Q = (Q_0,Q_1, \operatorname{t},\operatorname{h})$ such that
\begin{enumerate} [label=(\Alph*)]
 \item a cancellative dimer algebra $A' = kQ'/I'$ is obtained by contracting each arrow of $Q$ whose head has indegree 1; and
 \item for each $a \in Q_1$, the indegrees $\operatorname{deg}^+ \operatorname{t}(a)$ and $\operatorname{deg}^+ \operatorname{h}(a)$ are not both $1$.
\end{enumerate}

Set
\begin{equation*} \label{Qt}
Q_1^* = \left\{ a \in Q_1 \ | \ \operatorname{deg}^+\operatorname{h}(a) = 1 \right\} \ \ \ \text{ and } \ \ \ Q_1^{\operatorname{t}} := \left\{ a \in Q_1 \ | \ \operatorname{deg}^+\operatorname{t}(a) = 1 \right\}.
\end{equation*}
The quiver $Q' =(Q'_0,Q'_1,\operatorname{t}',\operatorname{h}')$ is then defined by
$$Q_0' = Q_0 / \left\{ \operatorname{h}(a) \sim \operatorname{t}(a) \ | \ a \in Q_1^* \right\}, \ \ \ \ Q_1' = Q_1 \setminus Q_1^*,$$
and for each arrow $a \in Q'_1$, 
$$\operatorname{h}'(a) = \operatorname{h}(a) \ \ \ \text{ and } \ \ \ \operatorname{t}'(a) = \operatorname{t}(a).$$ 

The homotopy algebras $A$ and $A'$ are isomorphic to tiled matrix rings.
Indeed, consider the $k$-linear map
$$\psi: A \rightarrow A'$$ 
defined by
$$\psi(a) = \left\{ \begin{array}{cl} a & \text{ if } \ a \in Q_0 \cup Q_1 \setminus Q_1^* \\ e_{\operatorname{t}(a)} & \text{ if } \ a \ \in Q_1^* \end{array} \right.$$
and extended multiplicatively to (nonzero) paths and $k$-linearly to $A$.
Furthermore, consider the polynomial ring generated by the simple matchings $\mathcal{S}'$ of $A'$, 
$$B = k\left[ x_D \ | \ D \in \mathcal{S}' \right].$$ 
By \cite[Theorem 1.1]{B2}, there are injective algebra homomorphisms
$$\tau: A' \hookrightarrow M_{|Q'_0|}(B) \ \ \ \text{ and } \ \ \ \tau_{\psi}: A \hookrightarrow M_{|Q_0|}(B)$$
defined by
\begin{align*}
\tau(a) = & \left\{ \begin{array}{ll}
e_{ii} & \text{ if } a = e_i \in Q'_0 \\
\left( \prod_{D \in \mathcal{S}' \, : \, D \ni a} x_D \right) e_{\operatorname{h}(a),\operatorname{t}(a)} & \text{ if } a \in Q'_1
\end{array} \right.
\\
\tau_{\psi}(a) = & \left\{ \begin{array}{ll}
e_{ii} & \text{ if } a = e_i \in Q_0 \\
\left( \prod_{D \in \mathcal{S}' \, : \, D \ni \psi(a)} x_D \right) e_{\operatorname{h}(a),\operatorname{t}(a)} & \text{ if } a \in Q_1
\end{array} \right.
\end{align*}
and extended multiplicatively and $k$-linearly to $A'$ and $A$.

For $p \in e_jAe_i$ and $p' \in e_jA'e_i$, denote by 
$$\bar{\tau}_{\psi}(p) = \overbar{p} \in B \ \ \ \text{ and } \ \ \ \bar{\tau}(p') = \overbar{p}' \in B$$ 
the single nonzero matrix entry of $\tau_{\psi}(p)$ and $\tau(p')$, respectively. 
Note that 
$$\bar{\tau}_{\psi}(p) = \bar{\tau} (\psi(p)).$$
Furthermore, for each $a \in Q_1$ and $D \in \mathcal{S}'$,
$$x_D | \overbar{a} \ \ \ \Longleftrightarrow \ \ \ \psi(a) \in D.$$

Since $A'$ is cancellative, each $a' \in Q'_1$ is contained in a simple matching by \cite[Theorem 1.1]{B4}; in particular, $\overbar{a}' \not = 1$.
Therefore, for each $a \in Q_1$,
$$\overbar{a} = 1 \ \ \ \Longleftrightarrow \ \ \ \operatorname{deg}^+ \operatorname{h}(a) = 1.$$

\begin{Lemma} \label{over} \
\begin{enumerate}
 \item The cycle algebras of $A$ and $A'$ are equal,\footnote{The map $\psi$ is therefore called a `cyclic contraction' \cite[Section 3]{B2}.}
\begin{equation*} \label{cyclic contraction}
k\left[ \cup_{i \in Q_0} \bar{\tau}_{\psi}\left(e_iAe_i\right) \right] = k\left[ \cup_{i \in Q'_0} \bar{\tau}\left( e_iA'e_i \right) \right] = S.
\end{equation*} 
 \item The center $Z'$ of $A'$ is isomorphic to $S$, and the center $Z$ of $A$ is isomorphic to the intersection
$$Z \cong k\left[ \cap_{i \in Q_0} \bar{\tau}_{\psi}(e_iAe_i) \right] = R.$$
 \item $S$ is a depiction of $R$.
 \item If the indegree of a vertex $i \in Q_0$ is at least $2$, then
 $$\bar{\tau}_{\psi}(e_iAe_i) = S.$$
In particular, for each arrow $a \in Q_1$,
$$\bar{\tau}_{\psi}(e_{\operatorname{t}(a)}Ae_{\operatorname{t}(a)}) = S \ \ \ \text{ or } \ \ \ \bar{\tau}_{\psi}(e_{\operatorname{h}(a)}Ae_{\operatorname{h}(a)}) = S.$$
\end{enumerate}
\end{Lemma}

\begin{proof}
(1) By assumption (A), for each cycle $p'$ in $Q'$, there is a cycle $p$ in $Q$ such that $\psi(p) = p'$.
Therefore the cycle algebras of $A$ and $A'$ are equal.

(2) Since $A'$ is cancellative, for each $i,j \in Q'_0$,
\begin{equation*} \label{what is time?}
\bar{\tau}(e_iA'e_i) = \bar{\tau}(e_jA'e_j),
\end{equation*}
by \cite[Theorem 1.1]{B4}.
Whence for each $i \in Q'_0$,
\begin{equation} \label{what is time?2}
\bar{\tau}(e_iA'e_i) = S.
\end{equation}
Furthermore, the centers $Z$ and $Z'$ are isomorphic to the intersections
$$Z \cong k\left[ \cap_{i \in Q_0} \bar{\tau}_{\psi}(e_iAe_i) \right]  = R \ \ \ \text{ and } \ \ \ Z' \cong k \left[ \cap_{i \in Q'_0} \bar{\tau}(e_iA'e_i) \right],$$
by \cite[Theorem 1.1]{B2}.
Therefore $Z'$ is isomorphic to $S$ by (\ref{what is time?2}).

(3) Since $A$ and $A'$ have equal cycle algebras, $Z \cong R$ is depicted by $Z' \cong S$, by \cite[Theorem 1.1]{B6}. 

(4) By assumption (A), if a vertex $i \in Q_0$ has indegree at least 2, then
\begin{equation*} \label{what is time?3}
\bar{\tau}_{\psi}(e_iAe_i) = \bar{\tau}(e_{\psi(i)}A'e_{\psi(i)}) \stackrel{\textsc{(i)}}{=} S,
\end{equation*}
where (\textsc{i}) holds by (\ref{what is time?2}).
Furthermore, by assumption (B), the head or tail of each arrow $a \in Q_1$ has indegree at least 2.
\end{proof}

\section{Prime decomposition of the origin}

Recall that $A$ is a nonnoetherian homotopy algebra with center $R$, satisfying assumptions (A) and (B) given in Section \ref{section 4}.
Consider the origin of $\operatorname{Max}R$,
$$\mathfrak{m}_0 := \left(x_D \ | \ D \in \mathcal{S}' \right)B \cap R.$$
For a monomial $g \in B$, denote by $\mathfrak{q}_g$ the ideal in $S$ generated by all monomials in $S$ that are divisible by $g$ in $B$.
If $g = x_D$ for some simple matching $D \in \mathcal{S}'$, then set
$$\mathfrak{q}_D := \mathfrak{q}_{x_D}.$$
We will write $h \mid g$ if $h$ divides $g$ in $B$, unless stated otherwise.

\begin{Lemma} \label{prime1}
Let $g \in B$ be a monomial.
Then the ideal $\mathfrak{q}_g \subset S$ is prime if and only if $g = x_D$ for some $D \in \mathcal{S}'$.
\end{Lemma}

\begin{proof}
Let $n := | \mathcal{S}' |$, and enumerate the simple matchings of $A'$, $\mathcal{S}' = \{ D_1, \ldots, D_n \}$.
Set $x_i := x_{D_i}$.

(i) We first claim that for each pair of distinct simple matchings $D_i,D_j \in \mathcal{S}'$, there is a cycle $s \in A$ satisfying
\begin{equation} \label{the end}
x_i \mid \overbar{s} \ \ \ \text{ and } \ \ \ x_j \nmid \overbar{s}.
\end{equation}

Indeed, fix $i \not = j$.
Since $D_i \not = D_j$, there is an arrow $a \in Q'_1$ for which $a \in D_i \setminus D_j$.
Furthermore, since $D_j$ is simple, there is a path $p \in e_{\operatorname{t}(a)}A'e_{\operatorname{h}(a)}$ supported on $Q'\setminus D_j$.
Whence $s:= pa$ is a cycle satisfying (\ref{the end}).
But $A$ and $A'$ have equal cycle algebras by Lemma \ref{over}.1.
Therefore $\overbar{s}$ is the $\bar{\tau}_{\psi}$-image of a cycle in $A$, proving our claim.

(ii) We now claim that if $g \in B$ is a monomial and $\mathfrak{q}_g$ is a prime ideal of $S$, then $g = x_D$ for some $D \in \mathcal{S}'$.
It suffices to consider a monomial $g = \prod_{i = 1}^{n'} x_i^{m_i}$, where $2 \leq n' \leq n$, and for each $i$, $m_i \geq 1$.
By Claim (i), there are cycles $s_1, \ldots, s_{n'} \in A$ such that 
$$x_1 \mid \overbar{s}_1, \ \ \ \ x_2 \nmid \overbar{s}_1,$$
and for each $2 \leq i \leq n'$,
$$x_1 \nmid \overbar{s}_i, \ \ \ \ x_i \mid \overbar{s}_i.$$

Set 
$$h_1 := \overbar{s}_1^{m_1} \ \ \ \text{ and } \ \ \ h_2 := \prod_{i = 2}^{n'} \overbar{s}_i^{m_i}.$$
Then $h_1h_2 \in \mathfrak{q}_g$.
But $h_1 \not \in \mathfrak{q}_g$ and $h_2 \not \in \mathfrak{q}_g$ since $x_2 \nmid h_1$ and $x_1 \nmid h_2$.
Therefore $\mathfrak{q}_g$ is not prime.

(iii) Finally, consider a simple matching $D \in \mathcal{S}'$.
If $s,t \in e_iAe_i$ are cycles for which $x_D \mid \overbar{st}$, then $x_D \mid \overbar{s}$ or $x_D \mid \overbar{t}$, since $B$ is the polynomial ring generated by $\mathcal{S}'$.
Therefore the ideal $\mathfrak{q}_{x_D}$ is prime.
\end{proof}

\begin{Lemma} \label{green}
Let $i,j \in Q_0$ and $D \in \mathcal{S}'$.
If $\operatorname{deg}^+ i \geq 2$, or $i$ is not the tail of an arrow $a \in Q_1^{\operatorname{t}}$ for which $x_D \mid \overbar{a}$, then there is a path $p \in e_jAe_i$ such that $x_D \nmid \overbar{p}$.
\end{Lemma}

\begin{proof} 
(i) First suppose $\operatorname{deg}^+i \geq 2$.
Since $D$ is simple, there is a path $q \in e_{\psi(j)}A'e_{\psi(i)}$ supported on $Q' \setminus D$; whence $x_D \nmid \overbar{q}$.
Furthermore, since $\operatorname{deg}^+ i \geq 2$, there is a path $p \in e_jAe_i$ such that $\psi(p) = q$, by assumption (A).
In particular, $x_D \nmid \overbar{q} = \overbar{p}$.

(ii) Now suppose $\operatorname{deg}^+i = 1$. 
Let $a \in Q_1^{\operatorname{t}}$ be such that $\operatorname{t}(a) = i$.
Then $\operatorname{deg}^+ \operatorname{h}(a) \geq 2$ by assumption (B).
Thus there is a path $t \in e_jAe_{\operatorname{h}(a)}$ for which $x_D \nmid \overbar{t}$, by Claim (i).
Therefore if $x_D \nmid \overbar{a}$, then the path $p:= ta \in e_jAe_i$ satisfies $x_D \nmid \overbar{p}$.
\end{proof}

\begin{Notation} \rm{
Denote by $\sigma_i$ the unit cycle at vertex $i \in Q_0$, and by
$$\sigma := \bar{\tau}_{\psi}(\sigma_i) = \prod_{D \in \mathcal{S}'} x_D$$
the common $\bar{\tau}_{\psi}$-image of each unit cycle in $Q$.
($\sigma$ is also the $\bar{\tau}$-image of each unit cycle in $Q'$.)
Furthermore, consider a covering map of the torus, $\pi: \mathbb{R}^2 \to T^2$, such that for some $i \in Q_0$,
$$\pi(\mathbb{Z}^2) = i.$$ 
Denote by
$$Q^+ := \pi^{-1}(Q) \subset \mathbb{R}^2$$
the covering quiver of $Q$. 
For each path $p$ in $Q$, denote by $p^+$ a path in $Q^+$ with tail in $[0,1) \times [0,1) \subset \mathbb{R}^2$ satisfying $\pi(p^+) = p$.
}\end{Notation}

\begin{Lemma} \label{one more}
Let $a \in A'$ be an arrow and let $s \in e_{\operatorname{t}(a)}A'e_{\operatorname{t}(a)}$ be a cycle satisfying $\overbar{a} \mid \overbar{s}$.
Then there is a path $p \in e_{\operatorname{t}(a)}A'e_{\operatorname{h}(a)}$ such that 
$$s = pa.$$
\end{Lemma}

\begin{proof}
We use the notation in \cite[Notation 2.1]{B3}.
Suppose the hypotheses hold.\footnote{This proof is similar to \cite[Claim (i) in proof of Lemma 2.4]{B3}.}
It suffices to assume $\sigma \nmid \overbar{s}$ by \cite[Lemma 2.1]{B2}.
Whence $s \in \hat{\mathcal{C}}$ by \cite[Lemma 4.8.3]{B2}.
Let $u \in \mathbb{Z}^2$ be such that $s \in \hat{\mathcal{C}}^u$.
Since $A'$ is cancellative, for each $i \in Q'_0$ we have
\begin{equation} \label{hmm}
\hat{\mathcal{C}}^u_i \not = \emptyset,
\end{equation}
by \cite[Proposition 4.10]{B2}.
Consider $t \in \hat{\mathcal{C}}^u_{\operatorname{h}(a)}$.
Then $\overbar{s} = \overbar{t}$ by \cite[Proposition 4.20.2]{B2}.

Now the paths $(as)^+$ and $(ta)^+$ bound a compact region 
$$\mathcal{R}_{as,ta} \subset \mathbb{R}^2.$$
Furthermore, since $A'$ is cancellative, if a cycle $p$ is formed from subpaths of cycles in $\hat{\mathcal{C}}^u$, then $p$ is in $\hat{\mathcal{C}}^u$, by \cite[Proposition 4.20.3]{B2}.
Therefore we may suppose that the interior of $\mathcal{R}_{as,ta}$ does not contain any vertices of $Q'^+$, by (\ref{hmm}).

Assume to the contrary that $s^+$ and $t^+$ do not intersect (modulo $I$). 
Then $a$ is contained in a simple matching $D$ of $A'$ such that $x_D \nmid \overbar{s}$, by \cite[Lemma 4.15]{B2}; see Figure \ref{figure for one more}.i.
In particular, $x_D \mid \overbar{a}$.
But by assumption, $\overbar{a} \mid \overbar{s}$.
Thus $x_D \mid \overbar{s}$, a contradiction.

Therefore $s^+$ and $t^+$ intersect at a vertex $i^+$; see Figure \ref{figure for one more}.ii. 
By assumption, $\sigma \nmid \overbar{s} = \overbar{t}$.
Whence $\sigma \nmid \overbar{as}$ and $\sigma \nmid \overbar{ta}$ since $\overbar{a} \mid \overbar{s} = \overbar{t}$.
Thus
$$\overbar{s}_1 = \overbar{t}_1 \overbar{a} \ \ \ \text{ and } \ \ \ \overbar{a}\overbar{s}_2 = \overbar{t}_2,$$
by \cite[Lemma 4.3]{B2}.
Consequently,
$$\overbar{s_2 t_1 a} = \overbar{s}_2 \overbar{s}_1 = \overbar{s}.$$
Therefore, since $\tau: A' \to M_{|Q'_0|}(B)$ is injective, we have
$$s_2t_1a = s.$$  
In particular, we may take $p = s_2t_1$.
\end{proof}

\begin{figure}
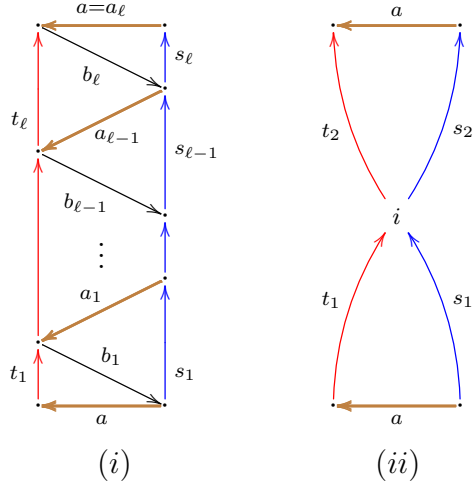

$$\begin{array}{ccc}
\xy 0;/r.4pc/:
(-5,-10)*{\cdot}="2";(5,-5)*{\cdot}="3";
(-5,-15)*{\cdot}="1";(5,-15)*{\cdot}="1'";
(5,0)*{\cdot}="6";(-5,5)*{\cdot}="5";
(0,-2.5)*{\vdots}="";
(5,10)*{\cdot}="7";
(5,-10)*{}="9";(-5,10)*{}="10";
(-5,15)*{\cdot}="8";(5,15)*{\cdot}="8'";
{\ar^{t_1}@[red]"1";"2"};{\ar^{b_1}"2";"1'"};{\ar@{-}_{s_1}@[blue]"1'";"9"};
{\ar^{a}@[brown]"1'";"1"};
{\ar@<.1mm>@[brown]"1'";"1"};
{\ar@<-.1mm>@[brown]"1'";"1"};
{\ar@<.09mm>@[brown]"1'";"1"};
{\ar@<-.09mm>@[brown]"1'";"1"};
{\ar_{a = a_{\ell}}@[brown]"8'";"8"};
{\ar@<.1mm>@[brown]"8'";"8"};
{\ar@<-.1mm>@[brown]"8'";"8"};
{\ar@<.09mm>@[brown]"8'";"8"};
{\ar@<-.09mm>@[brown]"8'";"8"};
{\ar^{a_{\ell-1}}@[brown]"7";"5"};
{\ar@<.1mm>@[brown]"7";"5"};
{\ar@<-.1mm>@[brown]"7";"5"};
{\ar@<.09mm>@[brown]"7";"5"};
{\ar@<-.09mm>@[brown]"7";"5"};
{\ar^{}@[red]"10";"8"};{\ar_{s_{\ell-1}}@[blue]"6";"7"};{\ar@{-}^{t_{\ell}}@[red]"5";"10"};
{\ar_{a_1}@[brown]"3";"2"};
{\ar@<.1mm>@[brown]"3";"2"};
{\ar@<-.1mm>@[brown]"3";"2"};
{\ar@<.09mm>@[brown]"3";"2"};
{\ar@<-.09mm>@[brown]"3";"2"};
{\ar^{}@[blue]"3";"6"};{\ar^{}@[blue]"9";"3"};{\ar^{}@[red]"2";"5"};
{\ar_{b_{\ell}}"8";"7"};{\ar_{b_{\ell-1}}"5";"6"};{\ar_{s_{\ell}}@[blue]"7";"8'"};
\endxy
& \ \ \ \ &
\xy 0;/r.4pc/:
(-5,15)*{\cdot}="8";(5,15)*{\cdot}="8'";
(-5,-15)*{\cdot}="1";(5,-15)*{\cdot}="1'";
(0,0)*+{\text{\scriptsize{$i$}}}="2";
{\ar@/^/^{t_2}@[red]"2";"8"};
{\ar@/^/^{t_1}@[red]"1";"2"};
{\ar@/_/_{s_2}@[blue]"2";"8'"};
{\ar@/_/_{s_1}@[blue]"1'";"2"};
{\ar_{a}@[brown]"8'";"8"};
{\ar@<.1mm>@[brown]"8'";"8"};
{\ar@<-.1mm>@[brown]"8'";"8"};
{\ar@<.09mm>@[brown]"8'";"8"};
{\ar@<-.09mm>@[brown]"8'";"8"};
{\ar^{a}@[brown]"1'";"1"};
{\ar@<.1mm>@[brown]"1'";"1"};
{\ar@<-.1mm>@[brown]"1'";"1"};
{\ar@<.09mm>@[brown]"1'";"1"};
{\ar@<-.09mm>@[brown]"1'";"1"};
\endxy
\\
(i) & & (ii)
\end{array}$$
\caption{Cases for Lemma \ref{one more}.  
In case (i), $s$ and $t$ factor into paths $s = s_{\ell} \cdots s_2s_1$ and $t = t_{\ell} \cdots t_2t_1$, where $a_1, \ldots, a_{\ell}, b_1, \ldots, b_{\ell}$ are arrows, and the cycles $b_ja_js_j$ and $a_{j-1}b_jt_j$ are unit cycles.
The $a_j$ arrows, drawn in thick brown, belong to a simple matching $D$ of $A'$.
In case (ii), $s$ and $t$ factor into paths $s = s_2e_is_1$ and $t = t_2 e_it_1$.}
\label{figure for one more}
\end{figure}

\begin{Proposition} \label{prime2}
For each arrow $a \in Q_1 \setminus Q_1^*$, $\bar{\tau}_{\psi}(e_{\operatorname{t}(a)}Aa)$ is an ideal of $S$ with prime decomposition
\begin{equation} \label{primary dec}
\bar{\tau}_{\psi}(e_{\operatorname{t}(a)}Aa) = \bigcap_{D \in \mathcal{S}' \, : \, x_D \mid \overbar{a}} \mathfrak{q}_D.
\end{equation}
Consequently, the prime decomposition of $\mathfrak{m}_0 \in \operatorname{Max}R$, as an ideal of $S$, is
\begin{equation*} \label{prime dec}
\mathfrak{m}_0 = \bigcap_{a \in Q_1^{\operatorname{t}}} \bar{\tau}_{\psi}(e_{\operatorname{t}(a)}Aa) = \bigcap_{\substack{D \in \mathcal{S}' : \\ x_D \mid \overbar{a} \text{ where } a \in Q_1^{\operatorname{t}}}} \mathfrak{q}_D.
\end{equation*}
\end{Proposition}

\begin{proof}
$\bar{\tau}_{\psi}(e_{\operatorname{t}(a)}Aa)$ is an ideal of $S$ by Lemma \ref{over}.4.
Set $\mathfrak{q}_a := \bigcap_{D \in \mathcal{S}' \, : \, x_D \mid \overbar{a}} \mathfrak{q}_D$.
The inclusion $\bar{\tau}_{\psi}(e_{\operatorname{t}(a)}Aa) \subseteq \mathfrak{q}_a$ is clear.
So suppose $t \in e_jAe_j$ is a cycle such that $\overbar{t} \in \mathfrak{q}_a$, that is, $\overbar{a} \mid \overbar{t}$.
We want to show that $\overbar{t} \in \bar{\tau}_{\psi}(e_{\operatorname{t}(a)}Aa)$.

First suppose $\deg^+ \operatorname{t}(a) \geq 2$.
Then $e_{\operatorname{t}(a)}Ae_{\operatorname{t}(a)} = Se_{\operatorname{t}(a)}$ by Lemma \ref{over}.4.
In particular, there is a cycle $s \in e_{\operatorname{t}(a)}Ae_{\operatorname{t}(a)}$ for which $\overbar{s} = \overbar{t}$.
Furthermore, there is a path $p \in e_{\operatorname{t}(a)}Ae_{\operatorname{h}(a)}$ such that $s = pa$, by Lemma \ref{one more} and assumption (A).

Now suppose $\deg^+ \operatorname{t}(a) = 1$.
Then $\deg^+ \operatorname{h}(a) \geq 2$ by assumption (B).
Whence $e_{\operatorname{h}(a)}Ae_{\operatorname{h}(a)} = Se_{\operatorname{h}(a)}$.
In particular, there is a cycle $s \in e_{\operatorname{h}(a)}Ae_{\operatorname{h}(a)}$ for which $\overbar{s} = \overbar{t}$.
Furthermore, there is a path $p \in e_{\operatorname{t}(a)}Ae_{\operatorname{h}(a)}$ such that $s = ap$, again by Lemma \ref{one more} and assumption (A).

Thus, in either case,  
$$\overbar{t} = \overbar{s} \in \bar{\tau}_{\psi}(e_{\operatorname{t}(a)}Aa).$$
Therefore (\ref{primary dec}) holds.
Finally, each $\mathfrak{q}_D$ is prime by Lemma \ref{prime1}. 
\end{proof}
 
In the following, we show that although the ideal $\mathfrak{q}_D$ may not be principal in $S$, it becomes principal over the localization $S_{\mathfrak{q}_D}$.

\begin{Proposition} \label{principal}
Let $D \in \mathcal{S}'$ and set $\mathfrak{q} := \mathfrak{q}_D$.
Then the maximal ideal $\mathfrak{q}S_{\mathfrak{q}}$ of $S_{\mathfrak{q}}$ is generated by $\sigma$,
$$\mathfrak{q} S_{\mathfrak{q}} = \sigma S_{\mathfrak{q}}.$$
\end{Proposition} 

\begin{proof}
Let $g \in \mathfrak{q}$ be a nonzero monomial.
Then there is a cycle $s \in A$ with $\overbar{s} = g$.
By possibly cyclically permuting the arrow subpaths of $s$, we may assume $s$ factors into paths $s = pa$, where $x_D \mid \overbar{a}$ and either 
\begin{itemize}
 \item[--] $a \in Q_1 \setminus \left( Q_1^* \cup Q_1^{\operatorname{t}} \right)$, or 
 \item[--] $a = a'\delta$ where $\delta \in Q_1^*$ and $a' \in Q_1^{\operatorname{t}}$.
\end{itemize}
In either case, $\operatorname{deg}^+ \operatorname{t}(a) \geq 2$.

Let $b$ be a path such that $ba$ is a unit cycle.
Then $x_D \nmid \overbar{b}$ since $x_D \mid \overbar{a}$ and $\overbar{ba} = \sigma$.
Furthermore, since $\operatorname{deg}^+ \operatorname{h}(b) = \operatorname{deg}^+ \operatorname{t}(a) \geq 2$, there is a path $t \in e_{\operatorname{t}(b)}Ae_{\operatorname{h}(b)}$ for which $x_D \nmid \overbar{t}$, by Lemma \ref{green}.
In particular, $tp$ and $tb$ are cycles, and $x_D \nmid \overbar{tb}$.
Whence 
$$\overbar{tp} \in S \ \ \ \text{ and } \ \ \ \overbar{tb} \in S \setminus \mathfrak{q}.$$
Therefore
$$g = \overbar{a} \overbar{p} \, \frac{\overbar{tb}}{\overbar{tb}} = \overbar{a} \overbar{b} \, \frac{\overbar{tp}}{\overbar{tb}} = \sigma \, \frac{\overbar{tp}}{\overbar{tb}} \in \sigma S_{\mathfrak{q}}.$$
\end{proof}

Recall that an ideal $I$ is unmixed if for each minimal prime $\mathfrak{q}$ over $I$, $\operatorname{ht}(\mathfrak{q}) = \operatorname{ht}(I)$.

\begin{Theorem} \label{height} \
\begin{enumerate} 
 \item For each $D \in \mathcal{S}'$, the height of $\mathfrak{q}_D$ in $S$ is 1.
 \item The set of minimal primes of $S$ over $\mathfrak{m}_0$ are the ideals $\mathfrak{q}_D \in \operatorname{Spec}S$ for which $D$ contains the $\psi$-image of some $a \in Q_1^{\operatorname{t}}$. 
 \item $\mathfrak{m}_0$ is an unmixed ideal of $S$.
Furthermore, $\mathfrak{m}_0$ has height $1$ as an ideal of $S$ and height $3$ as an ideal of $R$,
$$\operatorname{ht}_S(\mathfrak{m}_0) = 1 \ \ \ \text{ and } \ \ \ \operatorname{ht}_R(\mathfrak{m}_0) = 3.$$
\end{enumerate}
\end{Theorem}

\begin{proof}
(1) Set $\mathfrak{q} := \mathfrak{q}_D$.
Then
$$1 \stackrel{\textsc{(i)}}{\leq} \operatorname{ht}_S(\mathfrak{q}) = \operatorname{ht}_{S_{\mathfrak{q}}}(\mathfrak{q}S_{\mathfrak{q}}) \stackrel{\textsc{(ii)}}{=} \operatorname{ht}_{S_{\mathfrak{q}}}(\sigma S_{\mathfrak{q}}) \stackrel{\textsc{(iii)}}{\leq} 1.$$
Indeed, (\textsc{i}) holds since $S$ is an integral domain and $\mathfrak{q}$ is nonzero; (\textsc{ii}) holds by Proposition \ref{principal}; and (\textsc{iii}) holds by Krull's principal ideal theorem.

(2) Follows from Claim (1) and Proposition \ref{prime2}.

(3) $\mathfrak{m}_0$ is a height 1 unmixed ideal of $S$ by Claims (1) and (2), and Proposition \ref{prime2}.
Furthermore, $R$ admits a depiction by Lemma \ref{over}.3.
Thus the height of each maximal ideal of $R$ equals the Krull dimension of $R$ by \cite[Lemma 3.7.2]{B5}.
But the Krull dimension of $R$ is 3 by \cite[Theorem 1.1]{B6}.
Therefore $\operatorname{ht}_R(\mathfrak{m}_0) = 3$.
\end{proof}

\begin{Question} \rm{
Let $K$ be the function field of an algebraic variety.
As shown in Theorem \ref{height}.3, a subset $\mathfrak{p}$ of $K$ may be an ideal in different subalgebras of $K$, and the height of $\mathfrak{p}$ depends on the choice of such subalgebra.
Is the geometric height of $\mathfrak{p}$ independent of the choice of subalgebra for which $\mathfrak{p}$ is an ideal?
If this is the case, then the geometric height would be an intrinsic property of an ideal, whereas its height would not be.
}\end{Question}

The center and cycle algebra of $A_{\mathfrak{m}_0} := A \otimes_R R_{\mathfrak{m}_0}$ are respectively
\begin{equation*} \label{iloveaidan}
Z(A_{\mathfrak{m}_0}) \cong R \otimes_R R_{\mathfrak{m}_0} \cong R_{\mathfrak{m}_0} \ \ \ \text{ and } \ \ \ S \otimes_R R_{\mathfrak{m}_0} \cong SR_{\mathfrak{m}_0}.
\end{equation*}

\begin{Proposition} \label{ngd}
The cycle algebra $SR_{\mathfrak{m}_0}$ of $A_{\mathfrak{m}_0}$ is a normal Gorenstein domain.
\end{Proposition}

\begin{proof}
Let $\mathfrak{t} \in \operatorname{Spec}(SR_{\mathfrak{m}_0})$ and set $\mathfrak{q} := \mathfrak{t} \cap S$.

(i) We claim that 
$$(SR_{\mathfrak{m}_0})_{\mathfrak{t}} = S_{\mathfrak{q}}.$$
Clearly $(SR_{\mathfrak{m}_0})_{\mathfrak{t}} = S_{\mathfrak{q}}R_{\mathfrak{m}_0}$.\footnote{To show this, note that the elements of $SR_{\mathfrak{m}_0}$ are of the form $s/r$, with $s \in S$ and $r \in R \setminus \mathfrak{m}_0$.
Thus an element of $(SR_{\mathfrak{m}_0})_{\mathfrak{t}}$ is of the form $\frac{s_1}{r_1} (\frac{s_2}{r_2})^{-1}$, with $s_1,s_2 \in S$, $r_1, r_2 \in R \setminus \mathfrak{m}_0$, and $\frac{s_2}{r_2} \not \in \mathfrak{t}$.
Furthermore, $\frac{s_2}{r_2} \not \in \mathfrak{t}$ and (\ref{moving on}) together imply $s_2 \not \in \mathfrak{t}$.
Whence 
$$s_2 \in S \setminus (\mathfrak{t} \cap S) = S \setminus \mathfrak{q}.$$
Therefore
$$\frac{s_1}{r_1}\left(\frac{s_2}{r_2} \right)^{-1} = \frac{s_1r_2}{s_2} \cdot \frac{1}{r_1} \in S_{\mathfrak{q}}R_{\mathfrak{m}_0}.$$} 
It thus suffices to show that 
\begin{equation} \label{ilovekael}
S_{\mathfrak{q}}R_{\mathfrak{m}_0} = S_{\mathfrak{q}}.
\end{equation}
Indeed, we have
\begin{equation} \label{moving on}
\mathfrak{t} \cap R \subseteq \mathfrak{m}_0.
\end{equation}
Thus if $\mathfrak{m}_0 \subseteq \mathfrak{q}$, then $\mathfrak{q} \cap R = \mathfrak{m}_0$.
Whence $R_{\mathfrak{m}_0} \subseteq S_{\mathfrak{q}}$. 
In particular, $S_{\mathfrak{q}}R_{\mathfrak{m}_0} = S_{\mathfrak{q}}$.
Otherwise $\mathfrak{q} = 0 \subset \mathfrak{m}_0$ by Theorem \ref{height}.3; whence
$$S_{\mathfrak{q}}R_{\mathfrak{m}_0} = (\operatorname{Frac}S) R_{\mathfrak{m}_0} = \operatorname{Frac}S = S_{\mathfrak{q}}.$$
Therefore in either case (\ref{ilovekael}) holds, proving our claim.

(ii) $S$ is isomorphic to the center of $A'$ by Lemma \ref{over}.2.
Thus $S$ is a normal Gorenstein domain since $A'$ is an NCCR.
Whence $S_{\mathfrak{q}}$ is a normal Gorenstein domain.
But $(SR_{\mathfrak{m}_0})_{\mathfrak{t}} = S_{\mathfrak{q}}$ by Claim (i).
Therefore $(SR_{\mathfrak{m}_0})_{\mathfrak{t}}$ is a normal Gorenstein domain.
Since this holds for all $\mathfrak{t} \in \operatorname{Spec}(SR_{\mathfrak{m}_0})$, $SR_{\mathfrak{m}_0}$ is also a normal Gorenstein domain. 
\end{proof}

\section{Cycle regularity}

Recall that $A$ is a nonnoetherian homotopy algebra satisfying assumptions (A) and (B) given in Section \ref{section 4}, unless stated otherwise.
Let $\mathfrak{q} \in \operatorname{Spec}S$ be a minimal prime over the origin $\mathfrak{m}_0$ of $\operatorname{Max}R$; then there is a simple matching $D \in \mathcal{S}'$ such that $\mathfrak{q} = \mathfrak{q}_D$, by Proposition \ref{prime2}.
In this section, we will consider the cyclic localization $A_{\mathfrak{q}}$ of $A$ at $\mathfrak{q}$.

The algebra homomorphism $\tau_{\psi}: A \hookrightarrow M_{|Q_0|}(B)$ extends to the cyclic localization, $\tau_{\psi}: A_{\mathfrak{q}} \hookrightarrow M_{|Q_0|}(\operatorname{Frac}B)$. 
For $p \in e_jA_{\mathfrak{q}}e_i$, we will denote by $\bar{\tau}_{\psi}(p) = \overbar{p} \in \operatorname{Frac}B$ the single nonzero matrix entry of $\tau_{\psi}(p)$.

We begin by showing that a notion of homological regularity cannot be obtained by considering the central localization $A_{\mathfrak{m}_0} := A \otimes_R R_{\mathfrak{m}_0}$ alone.

\begin{Proposition} \label{infinite global dimension}
The $A_{\mathfrak{m}_0}$-module $A_{\mathfrak{m}_0}/\mathfrak{m}_0 = A \otimes_R (R_{\mathfrak{m}_0}/\mathfrak{m}_0)$ has infinite projective dimension, and therefore $A_{\mathfrak{m}_0}$ has infinite global dimension.
\end{Proposition}

\begin{proof}
By \cite[Lemmas 6.1 and 6.2]{B8}, there are monomials $g,h \in S$ such that for each $n \geq 1$, 
$$h^n \not \in R \ \ \ \text{ and } \ \ \ gh^n \in \mathfrak{m}_0 \subset R.$$
In particular, there is a vertex $i \in Q_0$ such that for each $n \geq 1$, 
$$h^n \not \in \bar{\tau}_{\psi}(e_iAe_i).$$

Let $s_n$ be the cycle in $e_iAe_i$ satisfying $\overbar{s}_n = gh^n$.
Consider a projective resolution of $A_{\mathfrak{m}_0}/\mathfrak{m}_0$ over $A_{\mathfrak{m}_0}$,
$$\cdots \rightarrow P_1 \longrightarrow A_{\mathfrak{m}_0} \stackrel{\cdot 1}{\longrightarrow} A_{\mathfrak{m}_0}/\mathfrak{m}_0 \to 0.$$
Each $s_n$ is in the zeroth syzygy module $\ker (\cdot 1) = \operatorname{ann}_{A_{\mathfrak{m}_0}}(A_{\mathfrak{m}_0}/\mathfrak{m}_0)$.
Thus $\ker(\cdot 1)$ is not finitely generated over $A_{\mathfrak{m}_0}$ since $h^n \not \in \bar{\tau}_{\psi}(e_iAe_i)$.
Furthermore, the cycles $s_n$ are pairwise commuting, and in particular there are an infinite number of independent commutation relations between them.
It follows that $\operatorname{pd}_{A_{\mathfrak{m}_0}}(A_{\mathfrak{m}_0}/\mathfrak{m}_0) = \infty$.
\end{proof}

\begin{Lemma} \label{simple dim}
Let $V$ be a simple $A_{\mathfrak{q}}$-module, and let $i \in Q_0$.
Then
\begin{equation*} \label{eiV}
\operatorname{dim}_k e_iV \leq 1.
\end{equation*}
\end{Lemma}

\begin{proof}
Suppose $V$ is a simple $A_{\mathfrak{q}}$-module.
Then $e_iV$ is a simple $e_iA_{\mathfrak{q}}e_i$-module.
Furthermore, the corner ring $e_iA_{\mathfrak{q}}e_i \cong \bar{\tau}_{\psi}(e_iA_{\mathfrak{q}}e_i) \subset B$ is a commutative $k$-algebra and $k$ is algebraically closed.
Therefore $\operatorname{dim}_k e_iV \leq 1$ by Schur's lemma.
\end{proof}

\begin{Lemma} \label{ann R}
Let $V$ be a simple $A_{\mathfrak{q}}$-module, and let $i \in Q_0$ be a vertex for which $e_iV \not = 0$.
Suppose $s \in e_iA_{\mathfrak{q}}e_i$.
Then $sV = 0$ if and only if $\overbar{s} \in \mathfrak{q}$. 
Consequently, $\operatorname{ann}_R V = \mathfrak{m}_0$.
\end{Lemma}

\begin{proof}
(i) Suppose $s \in e_iAe_i$ satisfies $\overbar{s} \in \mathfrak{q}$.
We claim that $sV = 0$.

Indeed, let $v \in e_iV$ be nonzero.
Then $\operatorname{dim}_k e_iV = 1$ by Lemma \ref{simple dim}.
Thus there is some $c \in k$ such that $(s-ce_i)e_iV = 0$.
Assume to the contrary that $c$ is nonzero.
Then $\overbar{s} - c \in S \setminus \mathfrak{q}$.
Therefore
$$v = \frac{s-ce_i}{\overbar{s}-c} \, v = \frac{1}{\overbar{s}-c} \, (s-ce_i)v = 0,$$
contrary to our choice of $v$. 

(ii) Conversely, suppose $s \in e_iAe_i$ satisfies $sV = 0$.
Assume to the contrary that $\overbar{s} \not \in \mathfrak{q}$; then $\overbar{s}^{-1} \in S_{\mathfrak{q}}$.
Whence
$$e_iV = \frac{s}{\overbar{s}} \, e_iV = \frac{1}{\overbar{s}} \, sV = 0,$$
contrary to our choice of vertex $i$.
\end{proof}

\begin{Definition} \rm{
Let $A$ be a ring with a complete set of orthogonal idempotents $\{e_1, \ldots, e_d \}$.
We say an element $p \in e_jAe_i$ is \textit{vertex invertible} if there is an element $p^* \in e_iAe_j$ such that 
$$p^*p = e_i \ \ \ \text{ and } \ \ \ pp^* = e_j.$$
Denote by $(e_jAe_i)^{\circ}$ the set of vertex invertible elements in $e_jAe_i$.
}\end{Definition}

For an arrow $a \in Q_1^{\operatorname{t}}$, denote by $\delta_a$ the unique arrow with $\operatorname{h}(\delta_a) = \operatorname{t}(a)$; in particular, $\delta_a \in Q_1^*$.

\begin{Lemma} \label{vertex invertible}
A path $p \in A$ is vertex invertible in $A_{\mathfrak{q}}$ if and only if $x_D \nmid \overbar{p}$ and the leftmost arrow subpath of $p$ is not an arrow $\delta_a \in Q_1^*$ for which $x_D \mid \overbar{a}$.
\end{Lemma}

\begin{proof}
(i) First suppose $x_D \mid \overbar{p}$.
Assume to the contrary that $p$ has vertex inverse $p^*$. 
Then
\begin{equation} \label{the-end}
p^* = \sum_{j=1}^m s_j^{-1} p_j
\end{equation}
for some $s_j \in S \setminus \mathfrak{q}$ and $p_j \in e_{\operatorname{t}(p)}Ae_{\operatorname{h}(p)}$.
In particular,
$$1 = \overbar{pp^*} = \overbar{p} \sum_j s_j^{-1}\overbar{p}_j.$$
Whence
$$s_1 \cdots s_m = \overbar{p} \sum_j \left( s_1 \cdots \hat{s}_j \cdots s_m \right) \overbar{p}_j \in B.$$
Thus $x_D \mid s_1 \cdots s_m$ since $x_D \mid \overbar{p}$.
Therefore $x_D \mid s_j$ for some $j$.
But then $s_j \in \mathfrak{q}$, a contradiction to our choice of $s_j$.

(ii) Now suppose the leftmost arrow subpath of $p$ is an arrow $\delta_a \in Q_1^*$ for which $x_D \mid \overbar{a}$.
If $p$ is a cycle, then $a$ is the rightmost arrow subpath of $p$.
Whence $x_D \mid \overbar{p}$.
Thus $p$ is not vertex invertible by Claim (i).

So suppose $p$ is not a cycle, and assume to the contrary that $p$ has vertex inverse $p^*$ given by (\ref{the-end}).
Since $p$ is not a cycle, we have $\operatorname{h}(p) \not = \operatorname{t}(p)$.
Thus each $p_j \in e_{\operatorname{t}(p)}Ae_{\operatorname{h}(p)}$ is a $k$-linear combination of nontrivial paths with tails at $\operatorname{h}(p)$.
But since $\operatorname{deg}^+ \operatorname{h}(p) = 1$, each nontrivial path $q \in A$ with tail at $\operatorname{h}(p)$ satisfies $x_D \mid \overbar{q}$.
Therefore $x_D$ divides each $\overbar{p}_j$ (in $B$). 
Furthermore, $x_D$ does not divide any $s_j$ since $s_j \in S \setminus \mathfrak{q}$. 
Whence $x_D \mid \overbar{p^*}$ in $BS_{\mathfrak{q}}$.
Thus $x_D \mid \overbar{p^*p}$ in $BS_{\mathfrak{q}}$, since $\overbar{p} \in B$.
Therefore $x_D \mid 1$ in $BS_{\mathfrak{q}}$.
But then $x_D$ is invertible in $BS_{\mathfrak{q}}$, a contradiction.

(iii) Finally suppose $x_D \nmid \overbar{p}$, and the leftmost arrow subpath of $\overbar{p}$ is not an arrow $\delta_a \in Q_1^*$ for which $x_D \mid \overbar{a}$.
Then there is a path $q \in e_{\operatorname{t}(p)}Ae_{\operatorname{h}(p)}$ satisfying $x_D \nmid \overbar{q}$, by Lemma \ref{green}.
Whence $pq$ is a cycle satisfying $x_D \nmid \overbar{pq}$; that is, $\overbar{pq} \in S \setminus \mathfrak{q}$.
Furthermore, $q$ has a vertex subpath $i$ for which $e_iAe_i = Se_i$, by Lemma \ref{over}.4.
Thus 
$$p^* := q (\overbar{pq})^{-1}$$ 
is in $A_{\mathfrak{q}}$.
But then
$$p^* p = \frac{q}{\overbar{pq}} \, p = \frac{\overbar{qp}}{\overbar{pq}} \, e_{\operatorname{t}(p)} = e_{\operatorname{t}(p)} \ \ \ \ \text{ and } \ \ \ \ pp^* = p \, \frac{q}{\overbar{pq}} = e_{\operatorname{h}(p)} \, \frac{\overbar{pq}}{\overbar{pq}} = e_{\operatorname{h}(p)}.$$
Therefore $p$ is vertex invertible in $A_{\mathfrak{q}}$.
\end{proof}

\begin{Lemma} \label{annihilate}
Let $V$ be a simple $A_{\mathfrak{q}}$-module.
\begin{enumerate}
 \item If $a \in Q_1 \setminus Q_1^*$ satisfies $x_D \mid \overbar{a}$, then $aV = 0$.
 \item If $\delta_a \in Q_1^*$ satisfies $x_D \mid \overbar{a}$, then $\delta_a V = 0$.
\end{enumerate}
\end{Lemma}

\begin{proof}
Let $a \in Q_1$ be an arrow for which $x_D \mid \overbar{a}$.

(i) First suppose $a \in Q_1 \setminus (Q_1^* \cup Q_1^{\operatorname{t}})$.
We claim that $aV = 0$.
Since $a \in Q_1 \setminus (Q_1^* \cup Q_1^{\operatorname{t}})$, there are paths 
$$s \in e_{\operatorname{h}(a)}Ae_{\operatorname{t}(a)} \ \ \ \text{ and } \ \ \ t \in e_{\operatorname{t}(a)}Ae_{\operatorname{h}(a)}$$ 
such that $x_D \nmid \overbar{s}$ and $x_D \nmid \overbar{t}$, by Lemma \ref{green}.
In particular, $x_D \nmid \overbar{st}$.
Whence
$$\overbar{st} \in S \setminus \mathfrak{q}.$$
Thus
$$a = \frac{st}{\overbar{st}} \, a = \frac{s}{\overbar{st}} \, ta \in A_{\mathfrak{q}} \mathfrak{q} e_{\operatorname{t}(a)}.$$
But $ta \in \mathfrak{q}e_{\operatorname{t}(a)} \cap e_{\operatorname{t}(a)}Ae_{\operatorname{t}(a)}$.
Therefore $a$ annihilates $V$ by Lemma \ref{ann R}.

(ii) Now suppose $a \in Q_1^{\operatorname{t}}$.
Set $\delta := \delta_a \in Q_1^*$.

(ii.a) We first claim that $a \delta V = 0$.
By assumption (B), $\operatorname{deg}^+ \operatorname{t}(\delta) \geq 2$ and $\operatorname{deg}^+ \operatorname{h}(a) \geq 2$.
Thus there are paths
$$s \in e_{\operatorname{h}(a)}Ae_{\operatorname{t}(\delta)} \ \ \ \text{ and } \ \ \ t \in e_{\operatorname{t}(\delta)}Ae_{\operatorname{h}(a)}$$ 
such that $x_D \nmid \overbar{s}$ and $x_D \nmid \overbar{t}$, by Lemma \ref{green}.
Whence
$$\overbar{st} \in S \setminus \mathfrak{q}.$$
Thus
$$a \delta = \frac{st}{\overbar{st}} \, a\delta = \frac{s}{\overbar{st}} \, ta\delta \in A_{\mathfrak{q}} \mathfrak{q} e_{\operatorname{t}(\delta)}.$$
Therefore $a \delta$ annihilates $V$ by Lemma \ref{ann R}.

(ii.b) We claim that $aV = 0$.
If $e_{\operatorname{t}(a)}V = 0$, then $aV = 0$, so suppose there is some nonzero $v \in e_{\operatorname{t}(a)}V$.
Assume to the contrary that $av \not = 0$.
Then, since $V$ is simple and $\operatorname{deg}^+ \operatorname{t}(a) = 1$, there is some $p \in A_{\mathfrak{q}}$ such that
$$w := \delta p a v \in e_{\operatorname{t}(a)}V$$
is nonzero.
By Claim (ii.a), $aw = (a \delta) (pav) = 0$.
Furthermore, $\operatorname{dim}_k e_{\operatorname{t}(a)}V = 1$ by Lemma \ref{simple dim}.
Thus, since $v,w \in e_{\operatorname{t}(a)}V$ are both nonzero, there is some $c \in k^*$ such that $cw = v$. 
But then 
$$0 \not = av = acw = c(aw) = 0,$$
which is not possible.

(ii.c) Finally, we claim that $\delta V = 0$.
Assume to the contrary that there is some $v \in e_{\operatorname{t}(\delta)}V$ such that $\delta v \not = 0$.
By Claim (2.i), $a \delta v = 0$.
But again $a$ is the only arrow with tail at $\operatorname{t}(a)$, and $\delta$ is not vertex invertible by Lemma \ref{vertex invertible}.
Therefore $V$ is not simple, a contradiction. 
\end{proof}

For each $\mathfrak{q}_D \in \operatorname{Spec}S$ minimal over $\mathfrak{m}_0$, set
$$\epsilon_D := 1_A - \sum_{a \in Q_1^{\operatorname{t}} \, : \, x_D \mid \overbar{a}} e_{\operatorname{t}(a)}.$$

\begin{Theorem} \label{simples}
Let $\mathfrak{q} = \mathfrak{q}_D \in \operatorname{Spec}S$ be minimal over $\mathfrak{m}_0 \in \operatorname{Max}R$.
Suppose there are $n$ arrows $a_1, \ldots, a_n \in Q_1^{\operatorname{t}}$ such that $x_D \mid \overbar{a}_{\ell}$.
Then there are precisely $n + 1$ non-isomorphic simple $A_{\mathfrak{q}}$-modules:
\begin{equation} \label{V0}
V_0 := A_{\mathfrak{q}} \epsilon_D/ A_{\mathfrak{q}} \mathfrak{q} \epsilon_D \cong \left( S_{\mathfrak{q}}/\mathfrak{q} \right)\epsilon_D,
\end{equation}
and for each $1 \leq \ell \leq n$, a vertex simple
\begin{equation} \label{Vell}
V_{\ell} := ke_{\operatorname{t}(a_{\ell})} \cong \left( R_{\mathfrak{m}_0}/\mathfrak{m}_0 \right) e_{\operatorname{t}(a_{\ell})}.
\end{equation}
\end{Theorem}

\begin{proof}
Let $V$ be a simple $A_{\mathfrak{q}}$-module. 
Let $a \in Q_1^{\operatorname{t}}$ be such that $x_D \mid \overbar{a}$. 
Then either $V$ is the vertex simple $V = ke_{\operatorname{t}(a)}$, or $e_{\operatorname{t}(a)}$ annihilates $V$, by Lemma \ref{annihilate}. 

So suppose $e_{\operatorname{t}(a)}V = 0$ for each $a \in Q_1^{\operatorname{t}}$ satisfying $x_D \mid \overbar{a}$.
We want to show that the sequence of left $A_{\mathfrak{q}}$-modules
$$0 \to A_{\mathfrak{q}}\mathfrak{q} \epsilon_D \longrightarrow A_{\mathfrak{q}} \epsilon_D \stackrel{g}{\longrightarrow} V \to 0$$
is exact.

We first claim that $g$ is onto.
Indeed, since $V \not = 0$, there is a vertex summand $e_i$ of $\epsilon_D$ for which $e_iV \not = 0$. 
Let $e_j$ be an arbitrary vertex summand of $\epsilon_D$.
Then there is a path $p \in e_jAe_i$ satisfying $x_D \nmid \overbar{p}$, by Lemma \ref{green}.
Thus, since $e_j$ is a summand of $\epsilon_D$, $p$ is vertex invertible by Lemma \ref{vertex invertible}.
Whence $e_jV \not = 0$ since $e_iV \not = 0$. 
Therefore $g$ is onto by Lemma \ref{simple dim}.

We now claim that the kernel of $g$ is $A_{\mathfrak{q}}\mathfrak{q} \epsilon_D$.
Let $b \in \epsilon_D A \epsilon_D$ be an arrow satisfying $bV = 0$.
Then there is a path $p \in e_{\operatorname{t}(b)}Ae_{\operatorname{h}(b)}$ satisfying $x_D \nmid \overbar{p}$, by Lemma \ref{green}.
Thus, since $e_{\operatorname{t}(b)}$ and $e_{\operatorname{h}(b)}$ are vertex summands of $\epsilon_D$, $p$ is vertex invertible in $A_{\mathfrak{q}}$ by Lemma \ref{vertex invertible}.
Whence
$$b = (p^*p)b = p^*(pb) \in A_{\mathfrak{q}}\mathfrak{q}\epsilon_D.$$
Thus the $A_{\mathfrak{q}}\epsilon_D$-annihilator of $V$ is $A_{\mathfrak{q}}\mathfrak{q}\epsilon_D$, by Lemma \ref{simple dim}.

Therefore $V = V_0$.
The simple modules $V_0, \ldots, V_n$ exhaust the possible simple $A_{\mathfrak{q}}$-modules, again by Lemma \ref{simple dim}.
\end{proof}

If $p \in A_{\mathfrak{q}}$ is a concatenation of paths and vertex inverses of paths in $A$, then we call $p$ a \textit{path}. 

\begin{Lemma} \label{free} 
Suppose $i \in Q_0$ satisfies $e_i \epsilon_D \not = 0$. 
Then for each $j \in Q_0$, the corner rings $e_jA_{\mathfrak{q}}e_i$ and $e_iA_{\mathfrak{q}}e_j$ are cyclic free $S_{\mathfrak{q}}$-modules.
Consquently, $A_{\mathfrak{q}}e_i$ and $e_iA_{\mathfrak{q}}$ are free $S_{\mathfrak{q}}$-modules. 
\end{Lemma}

\begin{proof} 
Suppose $e_i$ is a vertex summand of $\epsilon_D$.
Then either $e_iAe_i = Se_i$, or $i = \operatorname{t}(a)$ for some $a \in Q_1^{\operatorname{t}}$ with $x_D \nmid \overbar{a}$, by Lemma \ref{over}.4.
In the latter case, $a$ is vertex invertible by Lemma \ref{vertex invertible}, and $e_{\operatorname{h}(a)}Ae_{\operatorname{h}(a)} = Se_{\operatorname{h}(a)}$ by Lemma \ref{over}.4.
Thus in either case we have 
$$e_iA_{\mathfrak{q}}e_i = S_{\mathfrak{q}}e_i.$$
Therefore $A_{\mathfrak{q}}e_i$ and $e_iA_{\mathfrak{q}}$ are $S_{\mathfrak{q}}$-modules.

(i) We claim that for each $j \in Q_0$, $e_jA_{\mathfrak{q}}e_i$ is generated as an $S_{\mathfrak{q}}$-module by a single path; a similar argument holds for $e_iA_{\mathfrak{q}}e_j$.

(i.a) First suppose $j$ is not the tail of an arrow $a \in Q_1^{\operatorname{t}}$ for which $x_D \mid \overbar{a}$. 
Since $D \in \mathcal{S}'$ is a simple matching of $Q'$, there is path $s$ from $i$ to $j$ for which $x_D \nmid \overbar{s}$ (that is, $\psi(s)$ is supported on $Q' \setminus D$).
Thus $s$ has a vertex inverse $s^* \in e_iA_{\mathfrak{q}}e_j$, by Lemma \ref{vertex invertible}.

Let $t \in e_jA_{\mathfrak{q}}e_i$ be arbitrary.
Then $s^*t$ is in $e_iA_{\mathfrak{q}}e_i = S_{\mathfrak{q}}e_i$.
Whence
$$t = ss^*t \in s S_{\mathfrak{q}}.$$
Therefore $e_jA_{\mathfrak{q}}e_i = sS_{\mathfrak{q}}$.

(i.b) Now suppose $j$ is the tail of an arrow $a \in Q_1^{\operatorname{t}}$ for which $x_D \mid \overbar{a}$; in particular, $j \not = i$.
Since $D \in \mathcal{S}'$ is a simple matching of $Q'$, there is path $s$ from $i$ to $\operatorname{t}(\delta_a)$ for which $x_D \nmid \overbar{s}$.
Thus $s$ has a vertex inverse $s^* \in e_iA_{\mathfrak{q}}e_{\operatorname{t}(\delta_a)}$, again by Lemma \ref{vertex invertible}.

Let $t \in e_jA_{\mathfrak{q}}e_i$ be arbitrary.
Since $j \not = i$ and $\deg^+ j = 1$, there is some $r \in e_{\operatorname{t}(\delta_a)}A_{\mathfrak{q}}e_i$ satisfying $t = \delta_a r$.
Whence
$$t = \delta_a r = \delta_a s s^* r \in \delta_a s S_{\mathfrak{q}}.$$
Therefore $e_jA_{\mathfrak{q}}e_i = \delta_a s S_{\mathfrak{q}}$.

(ii) Finally, we claim that $e_jA_{\mathfrak{q}}e_i$ is a free $S_{\mathfrak{q}}$-module; a similar argument holds for $e_iA_{\mathfrak{q}}e_j$.
By Claim (i), there is a path $s$ such that
$$e_jA_{\mathfrak{q}}e_i = sS_{\mathfrak{q}}.$$
Furthermore, the $S_{\mathfrak{q}}$-module homomorphism
$$S_{\mathfrak{q}} \to sS_{\mathfrak{q}}, \ \ \ \ t \mapsto st,$$
is an isomorphism since $S_{\mathfrak{q}}$ and $\overbar{s}$ belong to the domain $\operatorname{Frac}B$, and $\bar{\tau}_{\psi}$ is injective.
\end{proof}

\begin{Lemma} \label{proj res}
The $A_{\mathfrak{q}}$-module $V_0$ satisfies
$$\operatorname{pd}_{A_{\mathfrak{q}}}\left(V_0 \right) \leq \operatorname{pd}_{S_{\mathfrak{q}}}\left(S_{\mathfrak{q}}/\mathfrak{q} \right).$$
\end{Lemma}

\begin{proof}
Consider a minimal free resolution of $S_{\mathfrak{q}}/\mathfrak{q}$ over $S_{\mathfrak{q}}$,
$$\cdots \to S_{\mathfrak{q}}^{\oplus n_1} \to S_{\mathfrak{q}} \to S_{\mathfrak{q}}/\mathfrak{q} \to 0.$$
Set $\epsilon := \epsilon_D$.
By Lemma \ref{free}, $A_{\mathfrak{q}}\epsilon$ is a free $S_{\mathfrak{q}}$-module.
Thus $A_{\mathfrak{q}}\epsilon$ is a flat $S_{\mathfrak{q}}$-module, that is, the functor $A_{\mathfrak{q}}\epsilon \otimes_{S_{\mathfrak{q}}} -$ is exact.
Therefore the sequence of left $A_{\mathfrak{q}}$-modules
\begin{equation} \label{proj res V0}
\cdots \to A_{\mathfrak{q}}\epsilon \otimes S_{\mathfrak{q}}^{\oplus n_1} \to A_{\mathfrak{q}}\epsilon \otimes S_{\mathfrak{q}} \to A_{\mathfrak{q}}\epsilon \otimes S_{\mathfrak{q}}/\mathfrak{q} \to 0
\end{equation}
is exact.
Each term is a projective $A_{\mathfrak{q}}$-module since
$$A_{\mathfrak{q}} \epsilon \otimes_{S_{\mathfrak{q}}} \left( S_{\mathfrak{q}}^{\oplus n_i} \right) \cong \left( A_{\mathfrak{q}} \epsilon \right)^{\oplus n_i}.$$
Furthermore, there is a left $A_{\mathfrak{q}}$-module isomorphism
$$V_0 = A_{\mathfrak{q}} \epsilon / A_{\mathfrak{q}} \mathfrak{q} \epsilon \cong A_{\mathfrak{q}}\epsilon \otimes_{S_{\mathfrak{q}}} S_{\mathfrak{q}}/\mathfrak{q}.$$
Therefore (\ref{proj res V0}) is a projective resolution of $V_0$ over $A_{\mathfrak{q}}$ of length at most $\operatorname{pd}_{S_{\mathfrak{q}}}\left(S_{\mathfrak{q}}/\mathfrak{q} \right)$.
\end{proof}

\begin{Lemma} \label{Sq regular}
The local ring $S_{\mathfrak{q}}$ is regular.
\end{Lemma}

\begin{proof}
$S$ is normal since $S$ is isomorphic to the center of the (noetherian) NCCR $A'$.
In particular, the singular locus of $\operatorname{Max}S$ has codimension at least 2.
Furthermore, the zero locus $\mathcal{Z}(\mathfrak{q})$ in $\operatorname{Max}S$ has codimension 1, by Theorem \ref{height}.1.
Therefore $\mathcal{Z}(\mathfrak{q})$ contains a smooth point of $\operatorname{Max}S$. 
\end{proof}

\begin{Proposition} \label{cycle hom}
Let $\mathfrak{q} \in \operatorname{Spec}S$ be minimal over $\mathfrak{m}_0$.
Then each simple $A_{\mathfrak{q}}$-module has projective dimension 1.
Consequently, for each simple $A_{\mathfrak{q}}$-module $V$,
$$\operatorname{pd}_{A_{\mathfrak{q}}}(V) = \operatorname{ht}_S(\mathfrak{q}).$$
\end{Proposition}

\begin{proof}
Recall the classification of simple $A_{\mathfrak{q}}$-modules given in Theorem \ref{simples}.

(i) Let $V_0$ be the simple $A_{\mathfrak{q}}$-module defined in (\ref{V0}). 
Then
$$1 \stackrel{\textsc{(i)}}{\leq} \operatorname{pd}_{A_{\mathfrak{q}}}\left(V_0 \right) \stackrel{\textsc{(ii)}}{\leq}  \operatorname{pd}_{S_{\mathfrak{q}}}\left(S_{\mathfrak{q}}/\mathfrak{q} \right) \stackrel{\textsc{(iii)}}{=} \operatorname{ht}_S(\mathfrak{q}) \stackrel{\textsc{(iv)}}{=} 1.$$
Indeed, (\textsc{i}) holds since $V_0$ is clearly not a direct summand of a free $A_{\mathfrak{q}}$-module; (\textsc{ii}) holds by Lemma \ref{proj res}; (\textsc{iii}) holds by Lemma \ref{Sq regular}; and (\textsc{iv}) holds by Theorem \ref{height}.1.

(ii) Fix $1 \leq \ell \leq n$, and let $V_{\ell}$ be the vertex simple $A_{\mathfrak{q}}$-module defined in (\ref{Vell}).
Set $a := a_{\ell}$.
We claim that $V_{\ell}$ has minimal projective resolution
\begin{equation} \label{blah}
0 \to A_{\mathfrak{q}}e_{\operatorname{h}(a)} \stackrel{\cdot a}{\longrightarrow} A_{\mathfrak{q}} e_{\operatorname{t}(a)} \stackrel{ \cdot 1}{\longrightarrow} k e_{\operatorname{t}(a)} = V_{\ell} \to 0.
\end{equation}

(ii.a) We first claim that $\cdot a$ is injective.
Suppose $b \in A_{\mathfrak{q}}e_{\operatorname{h}(a)}$ is nonzero.
Then $\bar{\tau}_{\psi}(ba) = \overbar{b} \cdot \overbar{a} \not = 0$ since $B$ is an integral domain.
Whence $ba \not = 0$ since $\bar{\tau}_{\psi}$ is injective.
Therefore $\cdot a$ is injective.

(ii.b) We now claim that $\operatorname{im}(\cdot a) = \ker(\cdot 1)$.
Since $aV = 0$, we have $\operatorname{im}(\cdot a) \subseteq \ker(\cdot 1)$.
To show the reverse inclusion, suppose $g \in \ker(\cdot 1)$; then $gV = 0$. 
We may write
$$g = \sum_j s_j^{-1}p_j,$$ 
where each $p_j \in Ae_{\operatorname{t}(a)}$ is a path and $s_j \in S \setminus \mathfrak{q}$.
If $p_j$ is nontrivial, then $p_j = p'_ja$ for some path $p'_j$ since $\deg^+ \operatorname{t}(a) = 1$.
Whence
$$p_jV_{\ell} = p'_ja V_{\ell} = 0.$$
It thus suffices to suppose that each $p_j$ is trivial, $p_j = e_{\operatorname{t}(a)}$.
But then $g = s^{-1} e_{\operatorname{t}(a)}$ for some $s \in S \setminus \mathfrak{q}$.
Therefore 
$$e_{\operatorname{t}(a)}V_{\ell} = sgV_{\ell} = 0,$$
a contradiction.

(ii.c) Finally, (\ref{blah}) is minimal since $V_{\ell}$ is clearly not a direct summand of a free $A_{\mathfrak{q}}$-module.
\end{proof}

Lemmas \ref{height lemma}, \ref{t}, and Proposition \ref{S regular local} are not specific to homotopy algebras.

\begin{Lemma} \label{height lemma}
Suppose $S$ is a depiction of $R$.
Let $\mathfrak{p} \in \operatorname{Spec}R$ and $\mathfrak{q} \in \iota_{S/R}^{-1}(\mathfrak{p})$.
If $\operatorname{ht}_S(\mathfrak{q}) = 1$, then $\operatorname{ght}_R(\mathfrak{p}) = 1$.
\end{Lemma}

\begin{proof}
Assume to the contrary that $\operatorname{ght}_R(\mathfrak{p}) = 0$.
Then there is a depiction $S'$ of $R$ and a prime ideal $\mathfrak{q}' \in \iota_{S'/R}^{-1}(\mathfrak{p})$ such that $\operatorname{ht}_{S'}(\mathfrak{q}') = 0$.
Whence $\mathfrak{q}' = 0$ since $S'$ is an integral domain. 
But then $\mathfrak{q}' \cap R = 0 \not = \mathfrak{q} \cap R = \mathfrak{p}$, a contradiction. 
Therefore
$$\operatorname{ht}_S(\mathfrak{q}) = 1 \leq \operatorname{ght}_R(\mathfrak{p}) \leq \operatorname{ht}_S(\mathfrak{q}).$$
\end{proof}

Recall that an ideal $I$ of an integral domain $S$ is a projective $S$-module if and only if $I$ is invertible, i.e., there is a fractional ideal $J$ such that $IJ = S$.
In this case, $I$ is a finitely generated rank one $S$-module \cite[Theorem 19.10]{C}.

\begin{Proposition} \label{S regular local}
Let $B$ be an integral domain, and let $A = [ A^{ij} ] \subset M_d(B)$ be a tiled matrix ring with cycle algebra $S$.
Set $Q_0 := \{ 1, \ldots, d \}$.
Suppose that
\begin{enumerate} 
 \item $S$ is a regular local ring.
 \item There is some $i \in Q_0$ such that 
   \begin{enumerate}
     \item $A^i = S$;
     \item for each $j \in Q_0$, $A^{ij}$ is an invertible ideal of $S$; and
     \item for each $j \in Q_0$, either $(e_iAe_j)^{\circ} \not = \emptyset$, or there is some $\ell \in Q_0$ and $b \in e_jAe_{\ell}$ satisfying
     $$e_jA = bA \oplus ke_j \ \ \ \text{ and } \ \ \ (e_iAe_{\ell})^{\circ} \not = \emptyset.$$
     \end{enumerate}
\end{enumerate}
Then
\begin{equation*} \label{gldimA}
\operatorname{gldim}A \leq \operatorname{dim}S.
\end{equation*}
\end{Proposition}

\begin{proof}
Suppose the hypotheses hold, and set $n := \operatorname{dim}S$.
Let $V$ be a left $A$-module.
We claim that 
$$\operatorname{pd}_A(V) \leq n.$$
It suffices to show that there is a projective resolution $P_{\bullet}$ of $V$,
$$\cdots \longrightarrow P_2 \stackrel{\delta_2}{\longrightarrow} P_1 \stackrel{\delta_1}{\longrightarrow} P_0 \stackrel{\delta_0}{\longrightarrow} V \to 0,$$
for which $\operatorname{ker}\delta_{n-1}$ is a projective $A$-module \cite[Proposition 8.6.iv]{R}. 

(i) We first claim that there is a projective resolution $P_{\bullet}$ of $V$ so that for each $\alpha \geq 1$,
\begin{equation} \label{ker}
\operatorname{ker} \delta_{\alpha} = Ae_i \ker \delta_{\alpha}.
\end{equation}

Indeed, fix $j \in Q_0$, and recall assumption (2.c).
If $p \in (e_iAe_j)^{\circ}$, then 
$$e_j \ker \delta_{\alpha} = p^*p \ker \delta_{\alpha} = p^* e_i p \ker \delta_{\alpha} \subseteq Ae_i \ker \delta_{\alpha}.$$
Otherwise there is some $\ell \in Q_0$ and $b \in e_jAe_{\ell}$ such that $e_jA = bA \oplus k e_j$ and $(e_iAe_{\ell})^{\circ} \not = \emptyset$. 
Let $p \in (e_iAe_{\ell})^{\circ}$.
Since the sum $e_j A = bA \oplus k e_j$ is direct, we may choose $P_{\bullet}$ so that for each $\alpha \geq 1$,
$$\delta_{\alpha} \mid_{e_jP_{\alpha}} = b \cdot \delta_{\alpha} \mid_{e_{\ell}P_{\alpha}}.$$
Furthermore, for nonzero $q \in e_{\ell}A$, $bq \not = 0$ since $B$ is an integral domain.
Thus 
$$e_j\ker \delta_{\alpha} = b \ker \delta_{\alpha}.$$
Whence
$$e_j\ker \delta_{\alpha} = b \ker \delta_{\alpha} = bp^*e_ip \ker \delta_{\alpha} \subseteq Ae_i \ker \delta_{\alpha}.$$
Therefore in either case, 
$$e_j\ker \delta_{\alpha} \subseteq Ae_i \ker \delta_{\alpha}.$$

(ii) Fix a projective resolution $P_{\bullet}$ of $V$ satisfying (\ref{ker}).
We claim that the left $A$-module $Ae_i \operatorname{ker}\delta_{n-1}$ is projective.

The right $A$-module $e_iA$ is projective, hence flat. 
Thus, setting $\otimes := \otimes_A$, the complex of $S$-modules
\begin{equation} \label{sequence}
\cdots \longrightarrow e_iA \otimes P_2 \stackrel{1 \otimes \delta_2}{\longrightarrow} e_iA \otimes P_1 \stackrel{1 \otimes \delta_1}{\longrightarrow} e_iA \otimes P_0 \stackrel{1 \otimes \delta_0}{\longrightarrow} e_iA \otimes V \to 0
\end{equation}
is exact.
Each term $e_iA \otimes P_{\ell}$ is a free $S$-module since
\begin{align*}
e_iA \otimes P_{\ell} \cong e_iA \otimes \bigoplus_j (Ae_j)^{\oplus n_j} \cong
\bigoplus_j (e_iA \otimes Ae_j )^{\oplus n_j} \notag\\
\cong \bigoplus_j (e_iAe_j)^{\oplus n_j} \cong \bigoplus_j (A^{ij})^{\oplus n_j} \stackrel{\textsc{(i)}}{\cong} \bigoplus_j S^{\oplus n_j},
\end{align*}
where (\textsc{i}) holds by assumption (2.b).
Furthermore, $e_iA \otimes V$ is an $S$-module since $e_iAe_i \cong S$ by assumption (2.a).
Therefore (\ref{sequence}) is a free resolution of an $S$-module.
But $\operatorname{gldim}S = \operatorname{dim}S = n$ by assumption (1).
Therefore the $n$th syzygy module of (\ref{sequence}) is a free $S$-module,
$$\operatorname{ker}(1 \otimes \delta_{n-1}) \cong S^{\oplus m}.$$

Since $e_iA$ is a flat right $A$-module, the sequence
$$0 \to e_iA \otimes \operatorname{ker}\delta_{n-1} \longrightarrow e_iA \otimes P_{n-1} \stackrel{1 \otimes \delta_{n-1}}{\longrightarrow} e_iA \otimes P_{n-2}$$
is exact.
Whence
$$e_iA \otimes \operatorname{ker}\delta_{n-1} \cong \operatorname{ker}(1 \otimes \delta_{n-1}) \cong S^{\oplus m}.$$
Therefore
$$Ae_i \operatorname{ker}\delta_{n-1} \cong Ae_iA \otimes \operatorname{ker} \delta_{n-1} \cong Ae_i S^{\oplus m} \stackrel{\textsc{(i)}}{\cong} A(e_iAe_i)^{\oplus m} \cong (Ae_i)^{\oplus m},$$
where (\textsc{i}) holds by assumption (2.a), proving our claim. 

(iii) Finally, $\operatorname{ker}\delta_{n-1}$ is a projective left $A$-module by Claims (i) and (ii).
Therefore $_AV$ has projective dimension at most $n$.
\end{proof}

\begin{Lemma} \label{t}
Suppose $S$ is a noetherian integral domain and a $k$-algebra, and $R$ is a subalgebra of $S$.
Let $\mathfrak{p} \in \operatorname{Spec}R$.
If $\mathfrak{t} \in \operatorname{Spec}(SR_{\mathfrak{p}})$ is a minimal prime over $\mathfrak{p}R_{\mathfrak{p}}$, then the ideal $\mathfrak{t} \cap S \in \operatorname{Spec}S$ is a minimal prime over $\mathfrak{p}$.
\end{Lemma}

\begin{proof}
Suppose that $\mathfrak{t} \cap S$ is not a minimal prime over $\mathfrak{p}$.
We want to show that $\mathfrak{t}$ is not a minimal prime over $\mathfrak{p}R_{\mathfrak{p}}$.
Since $\mathfrak{t} \cap S$ is not minimal, there is some $\mathfrak{q} \in \operatorname{Spec}S$, minimal over $\mathfrak{p}$, such that
\begin{equation} \label{searching}
\mathfrak{p} \subseteq \mathfrak{q} \subset \mathfrak{t} \cap S.
\end{equation}

(i) We claim that $\mathfrak{q} \cap R = \mathfrak{p}$.
Assume to the contrary that there is some $a \in (\mathfrak{t} \cap R) \setminus \mathfrak{p}$.
Then $a^{-1} \in R_{\mathfrak{p}}$.
Whence $1 = aa^{-1} \in \mathfrak{t}SR_{\mathfrak{p}} = \mathfrak{t}$, contrary to the fact that $\mathfrak{t}$ is prime.
Therefore 
\begin{equation} \label{searching2}
\mathfrak{t} \cap R \subseteq \mathfrak{p}.
\end{equation}
Consequently,
$$\mathfrak{p} \subseteq \mathfrak{q} \cap R \stackrel{\textsc{(i)}}{\subseteq} \mathfrak{t} \cap R \stackrel{\textsc{(ii)}}{\subseteq} \mathfrak{p},$$
where (\textsc{i}) holds by (\ref{searching}) and (\textsc{ii}) holds by (\ref{searching2}).
Thus $\mathfrak{q} \cap R = \mathfrak{p}$, proving our claim.

(ii) Now fix $a \in (\mathfrak{t} \cap S) \setminus \mathfrak{q}$, and assume to the contrary that $a \in \mathfrak{q}R_{\mathfrak{p}}$.
Then there is some $b \in \mathfrak{q}$ and $c \in R \setminus \mathfrak{p}$ such that $a = bc^{-1}$.
In particular, $ac = b \in \mathfrak{q}$.
Whence $c \in \mathfrak{q}$ since $c \in R \subseteq S$ and $\mathfrak{q}$ is prime.
Thus 
$$c \in \mathfrak{q} \cap R \stackrel{\textsc{(i)}}{=} \mathfrak{p},$$
where (\textsc{i}) holds by Claim (i).
But $c \not \in \mathfrak{p}$, a contradiction.
Whence $a \in \mathfrak{t} \setminus \mathfrak{q} R_{\mathfrak{p}}$. 
Thus
$$\mathfrak{p}R_{\mathfrak{p}} \subseteq \mathfrak{q}R_{\mathfrak{p}} \subset \mathfrak{t}.$$
Furthermore, $\mathfrak{q}R_{\mathfrak{p}}$ is a prime ideal of $SR_{\mathfrak{p}}$.
Therefore $\mathfrak{t}$ is a not a minimal prime over $\mathfrak{p}$.
\end{proof}

Again let $A$ be a nonnoetherian homotopy algebra satisfying assumptions (A) and (B).
Recall that the center and cycle algebra of $A_{\mathfrak{m}_0} := A \otimes_R R_{\mathfrak{m}_0}$ are isomorphic to $R_{\mathfrak{m}_0}$ and $SR_{\mathfrak{m}_0}$ respectively. 

\begin{Theorem} \label{big theorem 1}
$A_{\mathfrak{m}_0}$ is a noncommutative desingularization of its center.
Furthermore, for each $\mathfrak{t} \in \operatorname{Spec}(SR_{\mathfrak{m}_0})$ minimal over $\mathfrak{t} \cap R_{\mathfrak{m}_0}$,
$$\operatorname{gldim}A_{\mathfrak{t}} = \operatorname{dim}(SR_{\mathfrak{m}_0})_{\mathfrak{t}} = \operatorname{dim}S_{\mathfrak{t} \cap S}.$$
\end{Theorem}

\begin{proof}
By Lemma \ref{t} (with $\mathfrak{p} = \mathfrak{m}_0$), it suffices to consider prime ideals $\mathfrak{q} \in \operatorname{Spec}S$ which are minimal over $\mathfrak{m}_0$.

(i) \textit{$A_{\mathfrak{m}_0}$ is cycle regular.}
Let $\mathfrak{q} \in \operatorname{Spec}S$ be minimal over $\mathfrak{m}_0$, and let $V$ be a simple $A_{\mathfrak{q}}$-module.
The hypotheses of Proposition \ref{S regular local} hold: condition (1) holds by Lemma \ref{Sq regular}; (2.a) holds by Lemma \ref{over}.4; (2.b) holds by Lemma \ref{free}; and (2.c) holds by Lemma \ref{vertex invertible}.
Thus
$$1 \stackrel{\textsc{(i)}}{\leq} \operatorname{gldim}A_{\mathfrak{q}} \stackrel{\textsc{(ii)}}{\leq} \operatorname{dim}S_{\mathfrak{q}} = \operatorname{ht}_S(\mathfrak{q}) \stackrel{\textsc{(iii)}}{=} 1 \stackrel{\textsc{(iv)}}{=} \operatorname{ght}_R(\mathfrak{m}_0) \stackrel{\textsc{(v)}}{=} \operatorname{pd}_{A_{\mathfrak{q}}}(V).$$
Indeed, (\textsc{i}) and (\textsc{v}) hold by Proposition \ref{cycle hom}; (\textsc{ii}) holds by Proposition \ref{S regular local}; (\textsc{iii}) holds by Theorem \ref{height}.3; and (\textsc{iv}) holds by Lemma \ref{height lemma}.
Therefore $A_{\mathfrak{m}_0}$ is cycle regular.

(ii) \textit{$A_{\mathfrak{m}_0}$ is a noncommutative desingularization.}
By \cite[Corollary 2.14.1]{B3}, the (noncommutative) function fields of $A$ and $R$, and hence $A_{\mathfrak{m}_0}$ and $R_{\mathfrak{m}_0}$, are Morita equivalent,
$$A \otimes_R \operatorname{Frac}R \sim \operatorname{Frac}R.$$

(iii) Finally, suppose $\mathfrak{q} \in \operatorname{Spec}S$ is minimal over $\mathfrak{q} \cap R$.
We claim that $\operatorname{gldim}A_{\mathfrak{q}} = \operatorname{dim}S_{\mathfrak{q}}$.
By Theorem \ref{height}.2, either $\mathfrak{q} = \mathfrak{q}_D$ for some $D \in \mathcal{S}'$, or $\mathfrak{q} = 0$.
The case $\mathfrak{q} = \mathfrak{q}_D$ was shown in Claim (i), so suppose $\mathfrak{q} = 0$.

We first claim that for each $i \in Q_0$,
\begin{equation} \label{Frac S}
e_iA_{\mathfrak{q}}e_i = (\operatorname{Frac}S)e_i.
\end{equation}
Indeed, let $g \in \operatorname{Frac}S$ be arbitrary.
Fix $j \in Q_0$ for which $e_jAe_j = Se_j$.
Since $S$ is a domain,
\begin{equation} \label{frac S}
e_jA_{\mathfrak{q}}e_j = S_{\mathfrak{q}}e_j = (\operatorname{Frac}S)e_j.
\end{equation}
Thus there is an element $s \in e_jA_{\mathfrak{q}}e_j$ satisfying $\overbar{s} = g$.

Now fix a cycle $t_2e_jt_1 \in e_iA_{\mathfrak{q}}e_i$ that passes through $j$.
Then $t_1t_2 \in e_jA_{\mathfrak{q}}e_j$ has a vertex inverse $(t_1t_2)^*$ by (\ref{frac S}). 
Thus the element 
$$s' := t_2 (t_1t_2)^* st_1 \in e_iA_{\mathfrak{q}}e_i$$
satisfies $\overbar{s}' = \overbar{s} = g$.
Therefore (\ref{Frac S}) holds.

We now claim that for each $i,j \in Q_0$, there is a $(\operatorname{Frac}S)$-module isomorphism\footnote{In general, $\bar{\tau}_{\psi}(e_jAe_i)$ is not contained in $\operatorname{Frac}S$; otherwise (\ref{off diagonal}) would trivially hold.}
\begin{equation} \label{off diagonal}
e_jA_{\mathfrak{q}}e_i \cong \operatorname{Frac}S.
\end{equation}
Let $s \in e_jA_{\mathfrak{q}}e_i$ be arbitrary, and fix a cycle $t_2e_jt_1 \in e_iA_{\mathfrak{q}}e_i$ that passes through $j$. 
Then $t_1t_2$ has a vertex inverse $(t_1t_2)^*$ by (\ref{Frac S}).
Furthermore, $st_2 \in e_jA_{\mathfrak{q}}e_j$.
Thus
$$s = (t_1t_2)^*s (t_2t_1) \in (\operatorname{Frac}S)t_1.$$
Whence $e_jA_{\mathfrak{q}}e_i \subseteq (\operatorname{Frac}S)t_1$.
Conversely, (\ref{Frac S}) implies $e_jA_{\mathfrak{q}}e_i \supseteq (\operatorname{Frac}S)t_1$.
Thus
$$e_jA_{\mathfrak{q}}e_i = (\operatorname{Frac}S)t_1.$$
Furthermore, the $(\operatorname{Frac}S)$-module homomorphism 
$$\operatorname{Frac}S \to (\operatorname{Frac}S)t_1, \ \ \ \ s \mapsto st_1,$$
is an isomorphism since $\overbar{t}_1$ and $\operatorname{Frac}S$ are in the domain $\operatorname{Frac}B$, and $\bar{\tau}_{\psi}$ is injective.
Therefore (\ref{off diagonal}) holds.

It follows from (\ref{Frac S}) and (\ref{off diagonal}) that
$$A_{\mathfrak{q}} \cong M_d(\operatorname{Frac}S).$$
Thus $A_{\mathfrak{q}}$ is a semisimple algebra.
Therefore
$$\operatorname{gldim}A_{\mathfrak{q}} = 0 = \operatorname{dim}(\operatorname{Frac}S) = \operatorname{dim}S_{\mathfrak{q}}.$$
\end{proof}

\section{Local endomorphism rings}

Recall that $A$ is a nonnoetherian homotopy algebra satisfying assumptions (A) and (B) given in Section \ref{section 4}, unless stated otherwise.
For $a \in Q_1$, recall the ideal 
$$\mathfrak{m}_a := \bar{\tau}_{\psi}(e_{\operatorname{t}(a)}Aa) \subset S$$
from Proposition \ref{prime2}.
Given a simple matching $D \in \mathcal{S}'$ for which $\mathfrak{q} := \mathfrak{q}_D$ is a minimal prime over $\mathfrak{m}_0$, set
$$\mathfrak{m}_D := \bigcap_{a \in Q_1^{\operatorname{t}} \, : \, x_D \mid \overbar{a}} \mathfrak{m}_a \ \ \ \text{ and } \ \ \ \tilde{R} := \left(k+ \mathfrak{m}_D \right)_{\mathfrak{m}_D} + \mathfrak{q}S_{\mathfrak{q}}.$$

\begin{Lemma} \label{airplane}
Let $D \in \mathcal{S}'$ be a simple matching for which $\mathfrak{q} := \mathfrak{q}_D$ is a minimal prime over $\mathfrak{m}_0$, and let $a \in Q_1$.
If $x_D \mid \overbar{a}$, then
$$\mathfrak{m}_a S_{\mathfrak{q}} = \mathfrak{q} S_{\mathfrak{q}} = \sigma S_{\mathfrak{q}}.$$
\end{Lemma}

We note that the relation $\mathfrak{m}_a S_{\mathfrak{q}} = \mathfrak{q}S_{\mathfrak{q}}$ is nontrivial since if $\overbar{a} \not = x_D$, then $\mathfrak{q} \not \subseteq \mathfrak{m}_a$ in general; that is, there may be a cycle $s$ for which $x_D \mid \overbar{s}$ but $\overbar{a} \nmid \overbar{s}$.

\begin{proof}
Suppose $x_D \mid \overbar{a}$. 
Then
$$\sigma S_{\mathfrak{q}} \subseteq \bar{\tau}_{\psi}(e_{\operatorname{t}(a)}Aa) S_{\mathfrak{q}} = \mathfrak{m}_a S_{\mathfrak{q}} \subseteq \mathfrak{q} S_{\mathfrak{q}} \stackrel{\textsc{(i)}}{=} \sigma S_{\mathfrak{q}},$$
where (\textsc{i}) holds by Proposition \ref{principal}. 
\end{proof}

\begin{Proposition} \label{center}
Let $D \in \mathcal{S}'$ be a simple matching for which $\mathfrak{q} := \mathfrak{q}_D$ is a minimal prime over $\mathfrak{m}_0$.
The center $Z(A_{\mathfrak{q}})$ of $A_{\mathfrak{q}}$ is isomorphic to the subalgebra 
$$\tilde{R} := (k+ \mathfrak{m}_D)_{\mathfrak{m}_D} + \mathfrak{q}S_{\mathfrak{q}} = \bigcap_{a \in Q_1^{\operatorname{t}}} \bar{\tau}_{\psi}( e_{\operatorname{t}(a)}A_{\mathfrak{q}}e_{\operatorname{t}(a)}) \subset S_{\mathfrak{q}} \cong Z(A'_{\mathfrak{q}}).$$
\end{Proposition} 

\begin{proof}
Set 
$$Q_1^{\operatorname{t}} \cap D := Q_1^{\operatorname{t}} \cap \psi^{-1}(D) = \{ a \in Q_1^{\operatorname{t}} \ : \ x_D \mid \overbar{a} \}.$$
We claim that 
$$\begin{array}{rcl}
Z(A_{\mathfrak{q}}) & \stackrel{\textsc{(i)}}{\cong} & \bigcap_{i \in Q_0} \bar{\tau}_{\psi}(e_iA_{\mathfrak{q}}e_i)\\
& \stackrel{\textsc{(ii)}}{=} & \bigcap_{a \in Q_1^{\operatorname{t}}} \bar{\tau}_{\psi}(e_{\operatorname{t}(a)}A_{\mathfrak{q}}e_{\operatorname{t}(a)})\\
& \stackrel{\textsc{(iii)}}{=}  & \bigcap_{a \in Q_1^{\operatorname{t}}} \left( (k + \mathfrak{m}_a)_{\mathfrak{q} \cap (k + \mathfrak{m}_a)} + \mathfrak{m}_aS_{\mathfrak{q}} \right)\\
& \stackrel{\textsc{(iv)}}{=}  & \bigcap_{a \in Q_1^{\operatorname{t}} \cap D} \left( (k + \mathfrak{m}_a)_{\mathfrak{m}_a} + \mathfrak{q}S_{\mathfrak{q}} \right)\\
& \stackrel{\textsc{(v)}}{=}  & \bigcap_{a \in Q_1^{\operatorname{t}} \cap D} (k + \mathfrak{m}_a)_{\mathfrak{m}_a} + \mathfrak{q}S_{\mathfrak{q}}\\
& \stackrel{\textsc{(vi)}}{=} & (k + \cap_{a \in Q_1^{\operatorname{t}} \cap D} \mathfrak{m}_a)\left( \left( k + \cap_{a \in Q_1^{\operatorname{t}} \cap D} \mathfrak{m}_a \right) \setminus \cup_{a \in Q_1^{\operatorname{t}} \cap D} \mathfrak{m}_a \right)^{-1} + \mathfrak{q}S_{\mathfrak{q}}\\
& = & (k + \mathfrak{m}_D)_{\mathfrak{m}_D} + \mathfrak{q}S_{\mathfrak{q}}\\
& = & \tilde{R}.
\end{array}$$
Indeed, (\textsc{i}) holds by Lemma \ref{over}.2 and (\textsc{ii}) holds by Lemma \ref{over}.4.

To show (\textsc{iii}), suppose $a \in Q_1^{\operatorname{t}}$.
Recall the notation $A^{i} := \bar{\tau}_{\psi}(e_iAe_i)$.
Then 
$$A^{\operatorname{t}(a)} = k + \mathfrak{m}_a \ \ \ \text{ and } \ \ \ A^{\operatorname{h}(a)} = S.$$
Thus by the definition of cyclic localization,
\begin{align*}
\bar{\tau}_{\psi}\left( e_{\operatorname{t}(a)}A_{\mathfrak{q}}e_{\operatorname{t}(a)} \right) & = A^{\operatorname{t}(a)}_{\mathfrak{q} \cap A^{\operatorname{t}(a)}} + \sum_{\substack{qp \in e_{\operatorname{t}(a)}Ae_{\operatorname{t}(a)} \\ \text{a nontrivial cycle}}} \overbar{q} \,  A^{\operatorname{h}(p)}_{\mathfrak{q} \cap A^{\operatorname{h}(p)}} \, \overbar{p} \\
& = \left( k + \mathfrak{m}_a \right)_{\mathfrak{q} \cap \left(k + \mathfrak{m}_a \right)} + \sum_{\substack{q \in e_{\operatorname{t}(a)}Ae_{\operatorname{h}(a)} \\ \text{a path}}} \overbar{q} \, S_{\mathfrak{q}} \, \overbar{a} \\
& = \left( k + \mathfrak{m}_a \right)_{\mathfrak{q} \cap \left( k + \mathfrak{m}_a \right)} + \mathfrak{m}_a S_{\mathfrak{q}}.
\end{align*}

To show (\textsc{iv}), note that for $a \in Q_1^{\operatorname{t}}$, 
$$\mathfrak{m}_a \subseteq \mathfrak{q} \ \ \ \text{ if and only if } \ \ \ a \in \psi^{-1}(D).$$
Furthermore, if $\mathfrak{m}_a \subseteq \mathfrak{q}$, then $\mathfrak{m}_a S_{\mathfrak{q}} = \mathfrak{q}S_{\mathfrak{q}}$ by Lemma \ref{airplane}.  
Otherwise if $\mathfrak{m}_a \not \subseteq \mathfrak{q}$, then $\mathfrak{m}_aS_{\mathfrak{q}} = S_{\mathfrak{q}}$.

(\textsc{v}) holds since for $a \in Q_1^{\operatorname{t}} \cap D$,
$$\mathfrak{m}_a(k+ \mathfrak{m}_a)_{\mathfrak{m}_a} \subseteq \mathfrak{q}S_{\mathfrak{q}}.$$

Finally, to show (\textsc{vi}), recall that each $\mathfrak{m}_a$ is generated over $S$ by the $\bar{\tau}_{\psi}$-images of a set of nontrivial cycles, and thus by a set of nonconstant monomials in $S$.
Therefore for any $a,b \in Q_1^{\operatorname{t}}$, we have $(k+ \mathfrak{m}_a) \cap \mathfrak{m}_b = \mathfrak{m}_a \cap \mathfrak{m}_b$.
\end{proof}

\begin{Definition} \rm{
We say two arrows $a,b \in Q_1$ are \textit{coprime} if $\overbar{a}$ and $\overbar{b}$ are coprime in $B$; that is, the only common factors of $\overbar{a}$ and $\overbar{b}$ in $B$ are the units.
}\end{Definition}

\begin{Lemma} \label{center2}
Suppose the arrows in $Q_1^{\operatorname{t}}$ are pairwise coprime, and let $a \in Q_1^{\operatorname{t}}$. 
Consider a simple matching $D \in \mathcal{S}'$ for which $x_D \mid \overbar{a}$. 
Set $\mathfrak{q} := \mathfrak{q}_D$ and $i := \operatorname{t}(a)$. 
Then
$$Z(A_{\mathfrak{q}}) = \tilde{R} \, \mathbf{1} = A^i_{\mathfrak{q}} \, \mathbf{1} \cong e_iA_{\mathfrak{q}}e_i.$$
\end{Lemma}

\begin{proof}
Suppose the arrows in $Q_1^{\operatorname{t}}$ are pairwise coprime.
Then each arrow in $Q_1^{\operatorname{t}} \setminus \{ a\}$ is vertex invertible in $A_{\mathfrak{q}}$ by Lemma \ref{vertex invertible}.
Thus for each $j \in Q_0 \setminus \{ i\}$, 
$$e_jA_{\mathfrak{q}}e_j = S_{\mathfrak{q}}e_j,$$
by Lemma \ref{over}.4.
The lemma then follows by Proposition \ref{center}. 
\end{proof}

In the following two lemmas, let $B$ be an integral domain, and let $A = \left[ A^{ij} \right] \subset M_d(B)$ be a tiled matrix ring.
Fix $i,j,k \in \{ 1, \ldots, d \}$.
For $p \in e_iAe_j$, denote by $\overbar{p}$ the element of $B$ satisfying $p = \overbar{p}e_{ij}$.

\begin{Lemma} \label{someh}
Suppose 
\begin{equation} \label{Aij}
A^{ij} \not = 0, \ \ \ \ A^{ji} \not = 0,
\end{equation}
and 
\begin{equation} \label{Aij2}
A^i \mathbf{1}_d = Z(A).
\end{equation}
Then for each $f \in \operatorname{Hom}_{Z(A)}\left( e_jAe_i, e_kAe_i \right)$, there is some $h \in \operatorname{Frac}B$ such that for each $p \in e_jAe_i$, we have $\overbar{f(p)} = h \overbar{p}$.
\end{Lemma}

\begin{proof}
Let $f \in \operatorname{Hom}_{Z(A)}\left(e_jAe_i, e_k Ae_i \right)$.
By assumption (\ref{Aij}), there is some $0 \not = q \in e_iAe_j$.
By assumption (\ref{Aij2}), for $p_1,p_2 \in e_jAe_i$, 
$$\overbar{q} \, \overbar{p}_1 f(p_2) = \overbar{p_1q} f(p_2) = f((p_1q)p_2) = f(p_1(qp_2)) = f((p_2q)p_1) = \overbar{p_2q}f(p_1) = \overbar{q} \, \overbar{p}_2 f(p_1).$$
Thus, since $B$ is an integral domain,
\begin{equation*} \label{p1}
\overbar{p}_1 f(p_2) = \overbar{p}_2 f(p_1).
\end{equation*}
In particular, if $p_1$ and $p_2$ are nonzero, then
$$\frac{\overbar{f(p_1)}}{\overbar{p}_1} = \frac{\overbar{f(p_2)}}{\overbar{p}_2} =: h \in \operatorname{Frac}B.$$
Therefore for each $p \in e_jAe_i$, we have $\overbar{f(p)} = h \overbar{p}$. 
\end{proof}

\begin{Lemma} \label{endofstory}
Suppose (\ref{Aij}) and (\ref{Aij2}) hold. 
If there is some $p \in e_jAe_i$ such that for each $f \in \operatorname{Hom}_{Z(A)}\left( e_jAe_i, e_kAe_i \right)$, there is some $r \in e_kAe_j$ satisfying
\begin{equation} \label{Aij3}
f(p) = rp,
\end{equation}
then
$$\operatorname{Hom}_{Z(A)}\left(e_jAe_i, e_kAe_i \right) \cong e_kAe_j.$$
Similarly, if there is some $p \in e_iAe_j$ such that for each $f \in \operatorname{Hom}_{Z(A)}\left( e_iAe_j, e_iAe_k \right)$, there is some $r \in e_jAe_k$ satisfying $f(p) = pr$, then
$$\operatorname{Hom}_{Z(A)}\left( e_iAe_j, e_iAe_k \right) \cong e_jAe_k.$$
\end{Lemma}

\begin{proof}
Fix $f \in \operatorname{Hom}_{Z(A)}\left( e_jAe_i, e_kAe_i \right)$.
By Lemma \ref{someh}, there is some $h \in \operatorname{Frac}B$ such that for each $p \in e_jAe_i$, we have
\begin{equation} \label{wonderful}
\overbar{f(p)} = h \overbar{p}.
\end{equation}
Let $p'$ be as in (\ref{Aij3}).
Then there is some $r \in e_kAe_j$ such that $f(p') = rp'$.
Whence $\overbar{r} = h$ by (\ref{wonderful}), since $B$ is an integral domain.
Thus $r = he_{kj}$.
Therefore for each $p \in e_jAe_i$, we have $f(p) = rp$ by (\ref{wonderful}).  
Consequently, there is a surjective $Z(A)$-module homomorphism
\begin{equation} \label{ZA hom}
\begin{array}{ccc}
e_kAe_j & \twoheadrightarrow & \operatorname{Hom}_{Z(A)}\left(e_jAe_i, e_kAe_i \right)\\
r & \mapsto & (p \mapsto rp).
\end{array}
\end{equation}

To show injectivity, suppose $r,r' \in e_kAe_j$ are sent to the same homomorphism in $\operatorname{Hom}_{Z(A)}\left(e_jAe_i, e_kAe_i \right)$. 
Then for each $p \in e_jAe_i$,
$$rp = r'p.$$
But $e_jAe_i \not = 0$ by assumption (\ref{Aij}).
Whence $r = r'$ since $B$ is an integral domain.
Therefore (\ref{ZA hom}) is an isomorphism.

Similarly, there is a $Z(A)$-module isomorphism
$$\begin{array}{ccc}
e_jAe_k & \stackrel{\sim}{\longrightarrow} & \operatorname{Hom}_{Z(A)}\left(e_iAe_j, e_iAe_k \right)\\
r & \mapsto & (p \mapsto pr).
\end{array}$$
\end{proof}

Again let $A$ be a nonnoetherian homotopy algebra satisfying assumptions (A) and (B).
Furthermore, suppose the arrows in $Q_1^{\operatorname{t}}$ are pairwise coprime.
Fix $a \in Q_1^{\operatorname{t}}$, and consider a simple matching $D \in \mathcal{S}'$ such that $x_D \mid \overbar{a}$.
Set $\mathfrak{q} := \mathfrak{q}_D$ and $i := \operatorname{t}(a)$.

\begin{Lemma} \label{f in m0}
If $j \in Q_0$ is a vertex distinct from $i$ and $f \in \operatorname{Hom}_{\tilde{R}}\left( e_jA_{\mathfrak{q}}e_i, e_iA_{\mathfrak{q}}e_i\right)$, then 
$$\overbar{f(e_jA_{\mathfrak{q}}e_i)} \subseteq \mathfrak{m}_0 \tilde{R}.$$
\end{Lemma}

\begin{proof}
Fix a vertex $j \not = i \in Q_0$ and an $\tilde{R}$-module homomorphism $f: e_jA_{\mathfrak{q}}e_i \to e_iA_{\mathfrak{q}}e_i$. 
We may apply Lemma \ref{someh} to $f$: assumption (\ref{Aij}) holds since there is a path between any two vertices of $Q$, and assumption (\ref{Aij2}) holds by Lemma \ref{center2}. 
Thus there is some $h \in \operatorname{Frac}B$ such that for each $p \in e_jA_{\mathfrak{q}}e_i$, we have 
\begin{equation} \label{f(p)}
\overbar{f(p)} = h \overbar{p}.
\end{equation}

Assume to the contrary that there is some $p \in e_jA_{\mathfrak{q}}e_i$ such that $f(p) = ce_i + q$, where $0 \not = c \in k$ and $\overbar{q} \in \mathfrak{m}_0\tilde{R}$. 
By (\ref{f(p)}),
$$h \overbar{p} = \overbar{f(p)} = c + \overbar{q}.$$
Whence $h = (c + \overbar{q})\overbar{p}^{-1}$.

By assumption (A), there is a path $t' \in e_jAe_{\operatorname{h}(a)}$ such that (i) $x_D \nmid \overbar{t}'$, and (ii) $t'a$ is not a scalar multiple of $p$.
Set $t := t'a$.
Then \begin{equation} \label{not possible}
c\overbar{t}\overbar{p}^{-1} + \overbar{q}\overbar{t}\overbar{p}^{-1} = (c + \overbar{q}) \overbar{t}\overbar{p}^{-1} = h \overbar{t} \stackrel{(\textsc{i})}{=} \overbar{f(t)} \in \tilde{R} \stackrel{\textsc{(ii)}}{=} \bar{\tau}_{\psi}(e_iA_{\mathfrak{q}}e_i),
\end{equation}
where (\textsc{i}) holds by (\ref{f(p)}) and (\textsc{ii}) holds by Lemma \ref{center2}.
Furthermore, $\tilde{R}$ is a unique factorization domain since it is the localization of a subalgebra of the polynomial ring $B$ on a multiplicatively closed subset.
Thus, since $c \not = 0$, (\ref{not possible}) implies
\begin{equation} \label{not possible2}
\overbar{t}\overbar{p}^{-1} \in \bar{\tau}_{\psi}(e_iA_{\mathfrak{q}}e_i).
\end{equation}

Now every element $g \in \bar{\tau}_{\psi}(e_iA_{\mathfrak{q}}e_i)$ is of the form
\begin{equation} \label{dino}
g = d + \sum_{\ell = 1}^m x_D^{n_{\ell}} u_{\ell}v_{\ell}^{-1},
\end{equation}
where $d \in k$, and $u_{\ell},v_{\ell}$ are monomials in $B$ not divisible by $x_D$.
Moreover, for each $\ell$ we have $n_{\ell} \geq 1$, by Lemma \ref{vertex invertible}. 
The element $\overbar{t}\overbar{p}^{-1}$ is of the form (\ref{dino}), with $m \geq 1$ since $t$ is not a scalar multiple of $p$. 
But each $n_{\ell} \leq 0$ since $x_D \nmid \overbar{t}'$, contrary to (\ref{not possible2}).
\end{proof}

\begin{Proposition} \label{gameover} 
For each $j,k \in Q_0$,
\begin{equation*} \label{hom stuff}
\operatorname{Hom}_{\tilde{R}}\left(e_jA_{\mathfrak{q}}e_i, e_kA_{\mathfrak{q}}e_i \right) \cong e_kA_{\mathfrak{q}}e_j \ \ \ \text{ and } \ \ \ \operatorname{Hom}_{\tilde{R}}\left( e_iA_{\mathfrak{q}}e_j, e_iA_{\mathfrak{q}}e_k \right) \cong e_jA_{\mathfrak{q}}e_k.
\end{equation*}
\end{Proposition}

\begin{proof}
Suppose the hypotheses hold.
We claim that $A_{\mathfrak{q}}$ satisfies the assumptions of Lemma \ref{endofstory}, with $i = \operatorname{t}(a)$ and arbitrary $j,k \in Q_0$.

Indeed, assumption (\ref{Aij}) holds since there is a path between any two vertices of $Q$, and assumption (\ref{Aij2}) holds by Lemma \ref{center2}.

To show that the third assumption (\ref{Aij3}) holds, fix $j,k \in Q_0$.
Consider a path $p \in e_jAe_i$ for which $x_D^2 \nmid \overbar{p}$; such a path exists by assumption (A), and since $D$ is a simple matching of $A'$.
Let $f \in \operatorname{Hom}_{\tilde{R}}\left(e_jA_{\mathfrak{q}}e_i, e_kA_{\mathfrak{q}}e_i \right)$ be arbitrary.
We want to show that there is an $r \in e_kA_{\mathfrak{q}}e_j$ such that $f(p) = rp$.

Write $f(p) = \sum_{\ell} c_{\ell} q_{\ell}$ as an $\tilde{R}$-linear combination of paths $q_{\ell} \in e_kAe_i$.
To show that $f(p) = rp$, it suffices to show that for each path $q_{\ell}$, there is a path $r_{\ell}$ such that 
$$q_{\ell} = r_{\ell}p,$$ 
since then we may take $r = \sum_{\ell} c_{\ell} r_{\ell}$.
It therefore suffices to assume that $f(p) = q$ is a single path.

Let $p^+$ and $q^+$ be lifts of $p$ and $q$ to the covering quiver $Q^+$ with coincident tails, $\operatorname{t}(p^+) = \operatorname{t}(q^+) \in Q_0^+$.
Let $s \in e_kAe_j$ be a path for which $s^+$ has no cyclic subpaths in $Q^+$ and
$$\operatorname{t}(s^+) = \operatorname{h}(p^+) \ \ \ \text{ and } \ \ \ \operatorname{h}(s^+) = \operatorname{h}(q^+).$$
Then by \cite[Lemma 4.3]{B2}, there is some $n \in \mathbb{Z}$ such that
$$\overbar{sp} = \overbar{q} \sigma^n.$$

(i) First suppose $n \leq 0$. 
Set 
$$r := \sigma^n_k s.$$
Then $\overbar{rp} = \overbar{q}$.
Thus $rp = q$ since $\bar{\tau}_{\psi}$ is injective.

(ii) So suppose $n \geq 1$; without loss of generality we may assume $n = 1$.

(ii.a) Further suppose $i \not = k$ or $i = k \not = j$.
Then $q$ is a nontrivial path: if $i \not = k$, then $q$ is clearly nontrivial, and if $i = k \not = j$, then $q$ is nontrivial by Lemma \ref{f in m0}.

Since $\deg^+ i = 1$, $x_D$ divides the $\bar{\tau}_{\psi}$-image of each nontrivial path in $Ae_i$.
Whence $x_D \mid \overbar{q}$.
Thus $x_D^2 \mid \overbar{q}\sigma = \overbar{sp}$.
But $x_D^2 \nmid \overbar{p}$ by our choice of $p$.
Therefore $x_D \mid \overbar{s}$.
Consequently, $s$ factors into paths $s = s_3s_2s_1$, where $s_2$ is a subpath of a unit cycle satisfying $x_D \mid \overbar{s}_2$.
Let $b$ be one of the two paths for which $bs_2$ is a unit cycle.
Then $x_D \nmid \overbar{b}$ since $x_D \mid \overbar{s}_2$.
Thus $b$ has vertex inverse 
$$b^* \in e_{\operatorname{t}(s_3)}A_{\mathfrak{q}}e_{\operatorname{h}(s_1)},$$
by Lemma \ref{vertex invertible}.
Set 
$$r := s_3b^*s_1.$$
Then, since $\overbar{b^*} = \overbar{b}^{-1}$, we have
$$\overbar{rp} = \overbar{s_3b^*s_1 p} = \overbar{b}^{-1} \overbar{s}_3\overbar{s}_1 \overbar{p} = \frac{\overbar{s}_2}{\sigma} \overbar{s}_3\overbar{s}_1 \overbar{p} = \frac{\overbar{sp}}{\sigma} = \overbar{q}.$$ 
Therefore $rp = q$ since $\bar{\tau}_{\psi}$ is injective, proving our claim. 

(ii.b) Finally, suppose $i = j = k$.
Then $rp = f(p)$ holds by taking $p = e_i$ and $r = f(e_i)$.
\end{proof}

\begin{Theorem} \label{big theorem 2}
Suppose the arrows in $Q_1^{\operatorname{t}}$ are pairwise coprime.
Let $\mathfrak{q} \in \operatorname{Spec}S$ be a minimal prime over $\mathfrak{q} \cap R = \mathfrak{m}_0$.
Then there is some $i \in Q_0$ for which
$$A_{\mathfrak{q}} \cong \operatorname{End}_{Z(A_{\mathfrak{q}})}(A_{\mathfrak{q}}e_i).$$
Furthermore, $A_{\mathfrak{q}}e_i$ is a reflexive $Z(A_{\mathfrak{q}})$-module.
\end{Theorem}

\begin{proof}
Suppose the hypotheses hold.
By Theorem \ref{height}.2, there is some $D \in \mathcal{S}'$ such that $\mathfrak{q} = \mathfrak{q}_D$.
Since the arrows in $Q_1^{\operatorname{t}}$ are pairwise coprime, there is a unique arrow $a \in Q_1^{\operatorname{t}}$ for which $x_D \mid \overbar{a}$.
Set 
$$i := \operatorname{t}(a) \ \ \ \text{ and } \ \ \ \epsilon := \epsilon_D = 1_A - e_i.$$ 
For brevity, denote $\operatorname{Hom}_{\tilde{R}}(-,-)$ by $_{\tilde{R}}(-,-)$.
There are algebra isomorphisms
$$\begin{array}{rcl}
A_{\mathfrak{q}} & \cong & \left[ \begin{matrix} e_iA_{\mathfrak{q}}e_i & e_iA_{\mathfrak{q}}\epsilon \\ \epsilon A_{\mathfrak{q}} e_i & \epsilon A_{\mathfrak{q}} \epsilon \end{matrix} \right] \\
& \stackrel{\textsc{(i)}}{\cong} & \left[ \begin{matrix} _{\tilde{R}}(e_iA_{\mathfrak{q}}e_i, e_iA_{\mathfrak{q}}e_i) &  _{\tilde{R}}( \epsilon A_{\mathfrak{q}}e_i, e_i A_{\mathfrak{q}} e_i) \\ _{\tilde{R}}(e_iA_{\mathfrak{q}}e_i, \epsilon A_{\mathfrak{q}} e_i ) & _{\tilde{R}}(\epsilon A_{\mathfrak{q}}e_i, \epsilon A_{\mathfrak{q}} e_i ) \end{matrix} \right] \\
& \stackrel{\textsc{(ii)}}{\cong} & \operatorname{End}_{Z(A_{\mathfrak{q}})}( e_iA_{\mathfrak{q}}e_i \oplus \epsilon A_{\mathfrak{q}} e_i) \\
& = & \operatorname{End}_{Z(A_{\mathfrak{q}})}(A_{\mathfrak{q}}e_i),
\end{array}$$
where (\textsc{i}) holds by Proposition \ref{gameover} and (\textsc{ii}) holds by Lemma \ref{center2}.

Furthermore, $A_{\mathfrak{q}}e_i$ is a reflexive $Z(A_{\mathfrak{q}})$-module:
$$\begin{array}{rcl}
_{Z(A_{\mathfrak{q}})}( _{Z(A_{\mathfrak{q}})}(A_{\mathfrak{q}}e_i,Z(A_{\mathfrak{q}})),Z(A_{\mathfrak{q}})) & \stackrel{\textsc{(i)}}{=} & ((A_{\mathfrak{q}}e_i,e_iA_{\mathfrak{q}}e_i),e_iA_{\mathfrak{q}}e_i)\\ 
& \stackrel{\textsc{(ii)}}{=} & (e_iA_{\mathfrak{q}}, e_iA_{\mathfrak{q}}e_i) \\
& \stackrel{\textsc{(iii)}}{=} & A_{\mathfrak{q}}e_i,
\end{array}$$
where (\textsc{i}) holds by Lemma \ref{center2}, and (\textsc{ii}), (\textsc{iii}) hold by Proposition \ref{gameover}.
\end{proof}

\begin{Theorem} \label{NCCR theorem}
Let $A$ be a nonnoetherian homotopy algebra satisfying assumptions (A) and (B), and suppose the arrows in $Q_1^{\operatorname{t}}$ are pairwise coprime.
Then $A_{\mathfrak{m}_0}$ is a nonnoetherian NCCR.
\end{Theorem}

\begin{proof}
$A_{\mathfrak{m}_0}$ is nonnoetherian and an infinitely generated module over its nonnoetherian center by \cite[Section 3]{B6}; has a normal Gorenstein cycle algebra $SR_{\mathfrak{m}_0}$ by Proposition \ref{ngd}; is cycle regular by Theorem \ref{big theorem 1}; and for each prime $\mathfrak{q} \in \operatorname{Spec}(SR_{\mathfrak{m}_0})$ minimal over $\mathfrak{m}_0$, the cyclic localization $A_{\mathfrak{q}}$ is an endomorphism ring of a reflexive $Z(A_{\mathfrak{q}})$-module by Theorem \ref{big theorem 2}.
\end{proof}

\subsection{Examples} \label{Examples}

\begin{Example} \label{first example} \rm{
Set
$$B := k\left[x,y,z,w \right], \ \ \ \ S := k\left[ xz, yz, xw, yw \right] \cong k\left[ a,b,c,d \right]/(ad - bc),$$
and
$$I := (x,y)S, \ \ \ \ J := (z,w)S, \ \ \ \ \mathfrak{m}_0 := zI, \ \ \ \ R := k + \mathfrak{m}_0.$$
Consider the contraction of homotopy algebras given in Figure \ref{first figure}.
Each arrow is labeled by its $\bar{\tau}_{\psi}$/$\bar{\tau}$-image in $B$.
The center and cycle algebra of $A$ are $R$ and $S$ respectively. 

In this example, the maximal ideal $\mathfrak{m}_0 \in \operatorname{Max}R$ at the origin is a height one prime ideal of $S$.\footnote{Note that the ideals $xzS$ and $yzS$, each of which is properly contained in $zI$, are not prime since $(xw) \cdot (yz) \in xzS$ and $(xz) \cdot (yw) \in yzS$.}
Therefore $\mathfrak{m}_0$ itself is the only minimal prime of $S$ over $\mathfrak{m}_0$. 
Furthermore, the cyclic localization of $A$ at $\mathfrak{m}_0$ is
$$A_{\mathfrak{m}_0} = \left\langle \left[ \begin{matrix} S_{\mathfrak{m}_0} & I & zI \\ 
J & S_{\mathfrak{m}_0} & zS \\
S & I & R_{\mathfrak{m}_0} \end{matrix} \right] \right\rangle = \left[ \begin{matrix} 
S_{\mathfrak{m}_0} & IS_{\mathfrak{m}_0} & zIS_{\mathfrak{m}_0} \\ 
JS_{\mathfrak{m}_0} & S_{\mathfrak{m}_0} & zS_{\mathfrak{m}_0} \\
S_{\mathfrak{m}_0} & IS_{\mathfrak{m}_0} & R_{\mathfrak{m}_0} + \mathfrak{m}_0S_{\mathfrak{m}_0}\end{matrix} \right],$$
with center $Z(A_{\mathfrak{m}_0}) \cong R_{\mathfrak{m}_0} + \mathfrak{m}_0S_{\mathfrak{m}_0}$.
}\end{Example}

\begin{figure}
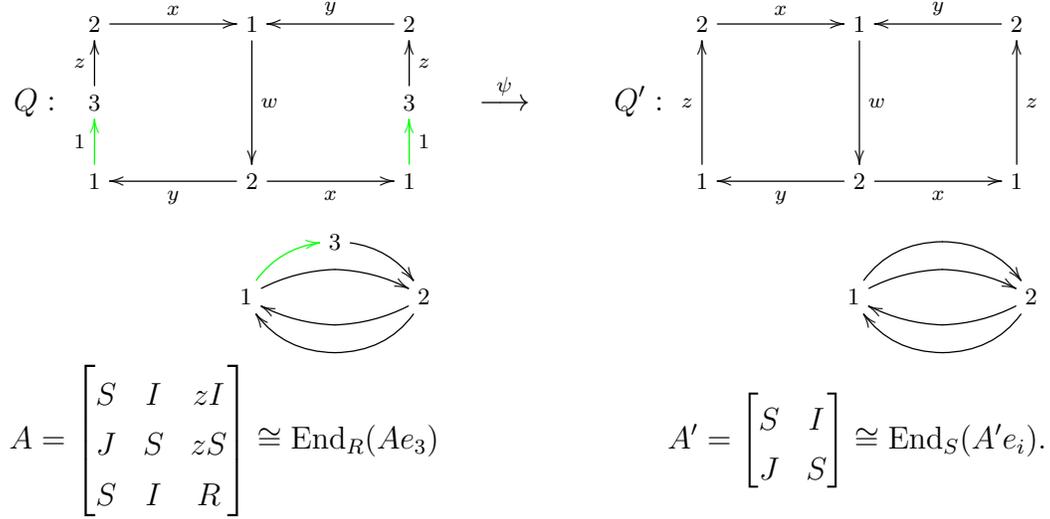

\begin{align*}
Q: \xy 0;/r.31pc/:
(-16,8)*+{\text{\scriptsize{$2$}}}="1";
(0,8)*+{\text{\scriptsize{$1$}}}="2";
(16,8)*+{\text{\scriptsize{$2$}}}="3";
(-16,-8)*+{\text{\scriptsize{$1$}}}="4";
(0,-8)*+{\text{\scriptsize{$2$}}}="5";
(16,-8)*+{\text{\scriptsize{$1$}}}="6";
(-16,0)*+{\text{\scriptsize{$3$}}}="7";
(16,0)*+{\text{\scriptsize{$3$}}}="8";
{\ar^x"1";"2"};{\ar_y"3";"2"};
{\ar^y"5";"4"};{\ar_x"5";"6"};
{\ar@[green]^1"4";"7"};{\ar^z"7";"1"};
{\ar@[green]_1"6";"8"};{\ar_z"8";"3"};
{\ar^w"2";"5"};
\endxy
& \ \ \ \stackrel{\psi}{\longrightarrow} \ \ \ &
Q': \xy 0;/r.31pc/:
(-16,8)*+{\text{\scriptsize{$2$}}}="1";
(0,8)*+{\text{\scriptsize{$1$}}}="2";
(16,8)*+{\text{\scriptsize{$2$}}}="3";
(-16,-8)*+{\text{\scriptsize{$1$}}}="4";
(0,-8)*+{\text{\scriptsize{$2$}}}="5";
(16,-8)*+{\text{\scriptsize{$1$}}}="6";
{\ar^x"1";"2"};{\ar_y"3";"2"};
{\ar^y"5";"4"};{\ar_x"5";"6"};
{\ar^z"4";"1"};
{\ar_z"6";"3"};
{\ar^w"2";"5"};
\endxy
\\
\ \ \ \ \ \ \ \xy 0;/r.35pc/:
(-8,0)*+{\text{\scriptsize{$1$}}}="1";
(8,0)*+{\text{\scriptsize{$2$}}}="2";
(0,5)*+{\text{\scriptsize{$3$}}}="3";
(0,2.5)*{}="4";
(0,-2.5)*{}="5";
(0,-5)*{}="6";
{\ar@/^/@[green]"1";"3"};
{\ar@/^/"3";"2"};
{\ar@{-}@/^.2pc/"1";"4"};
{\ar@/^.2pc/"4";"2"};
{\ar@{-}@/^.2pc/"2";"5"};
{\ar@/^.2pc/"5";"1"};
{\ar@{-}@/^/"2";"6"};
{\ar@/^/"6";"1"};
\endxy
&&
\ \ \ \ \ \ \ \xy 0;/r.35pc/:
(-8,0)*+{\text{\scriptsize{$1$}}}="1";
(8,0)*+{\text{\scriptsize{$2$}}}="2";
(0,5)*{}="3";
(0,2.5)*{}="4";
(0,-2.5)*{}="5";
(0,-5)*{}="6";
{\ar@{-}@/^/"1";"3"};
{\ar@/^/"3";"2"};
{\ar@{-}@/^.2pc/"1";"4"};
{\ar@/^.2pc/"4";"2"};
{\ar@{-}@/^.2pc/"2";"5"};
{\ar@/^.2pc/"5";"1"};
{\ar@{-}@/^/"2";"6"};
{\ar@/^/"6";"1"};
\endxy
\\
A = \left[ \begin{matrix}
S & I & zI \\ 
J & S & zS \\
S & I & R
\end{matrix} \right] \cong \operatorname{End}_R( Ae_3)
& & A' = \left[ \begin{matrix} S & I \\ J & S \end{matrix} \right] \cong \operatorname{End}_S(A'e_i).
\end{align*}
\caption{(Example \ref{first example}.)  The homotopy algebra $A$ is a nonnoetherian NCCR.
The quivers $Q$ and $Q'$ on the top line are each drawn on a torus, and the contracted arrow of $Q$ is drawn in green.}
\label{first figure}
\end{figure}

\begin{Example} \label{second example} \rm{
Set
$$B := k\left[x,y,z,w \right], \ \ \ \ S := k\left[ xz, yz, xw, yw \right],$$
and
$$I := (x,y)S, \ \ \ \ J := (z,w)S, \ \ \ \ \mathfrak{m}_0 := zwI^2, \ \ \ \ R := k + \mathfrak{m}_0.$$
Consider the contraction of homotopy algebras given in Figure \ref{second figure}.
As in Example \ref{first example}, the center and cycle algebra of $A$ are $R$ and $S$ respectively. 

The minimal primes in $S$ over $\mathfrak{m}_0$ are
$$\mathfrak{q}_1 := zI \ \ \ \text{ and } \ \ \ \mathfrak{q}_2 := wI,$$
each of height 1.
The cyclic localizations of $A$ at $\mathfrak{q}_1$ and $\mathfrak{q}_2$ are
$$A_{\mathfrak{q}_1} = \left[ \begin{matrix}
S_{\mathfrak{q}_1} & IS_{\mathfrak{q}_1} & \mathfrak{q}_1S_{\mathfrak{q}_1} & S_{\mathfrak{q}_1} \\ 
wS_{\mathfrak{q}_1} & S_{\mathfrak{q}_1} & zS_{\mathfrak{q}_1} & w S_{\mathfrak{q}_1} \\
S_{\mathfrak{q}_1} & IS_{\mathfrak{q}_1} & (k+\mathfrak{q}_1)_{\mathfrak{q}_1} + \mathfrak{q}_1S_{\mathfrak{q}_1} & S_{\mathfrak{q}_1}\\
S_{\mathfrak{q}_1} & IS_{\mathfrak{q}_1} & \mathfrak{q}_1S_{\mathfrak{q}_1} & S_{\mathfrak{q}_1}
\end{matrix} \right] \cong \operatorname{End}_{Z(A_{\mathfrak{q}_1})}(A_{\mathfrak{q}_1}e_3)$$
and
$$A_{\mathfrak{q}_2} = \left[ \begin{matrix}
S_{\mathfrak{q}_2} & I S_{\mathfrak{q}_2} & S_{\mathfrak{q}_2} & \mathfrak{q}_2S_{\mathfrak{q}_2} \\
zS_{\mathfrak{q}_2} & S_{\mathfrak{q}_2} & zS_{\mathfrak{q}_2} & wS_{\mathfrak{q}_2} \\
S_{\mathfrak{q}_2} & IS_{\mathfrak{q}_2} & S_{\mathfrak{q}_2} & \mathfrak{q}_2S_{\mathfrak{q}_2} \\
 S_{\mathfrak{q}_2} & IS_{\mathfrak{q}_2} & S_{\mathfrak{q}_2} & (k+ \mathfrak{q}_2)_{\mathfrak{q}_2} + \mathfrak{q}_2S_{\mathfrak{q}_2}
 \end{matrix} \right] \cong \operatorname{End}_{Z(A_{\mathfrak{q}_2})}(A_{\mathfrak{q}_2}e_4),$$
with respective centers 
$$Z(A_{\mathfrak{q}_1}) \cong (k+ \mathfrak{q}_1)_{\mathfrak{q}_1} + \mathfrak{q}_1S_{\mathfrak{q}_1} \ \ \ \ \text{ and } \ \ \ \ 
Z(A_{\mathfrak{q}_2}) \cong (k+ \mathfrak{q}_2)_{\mathfrak{q}_2} + \mathfrak{q}_2S_{\mathfrak{q}_2}.$$
(Note that $wS_{\mathfrak{q}_1} = JS_{\mathfrak{q}_1}$ since $z = w \, \frac{xz}{xw}$, and similarly $zS_{\mathfrak{q}_2} = JS_{\mathfrak{q}_2}$.)
In contrast to Example \ref{first example}, $A$ itself is not an endomorphism ring, although its cyclic localizations are.
}\end{Example}

\ \\
\textbf{Acknowledgments.}  The author would like to thank Kenny Brown for useful discussions, and an anonymous referee for helpful comments. 
This article was completed while the author was a research fellow at the Heilbronn Institute for Mathematical Research at the University of Bristol. 

\bibliographystyle{hep}

\begin{thebibliography}{10} 
\bibitem[AB]{AB} M.\ Auslander and D.\ A.\ Buchsbaum, Homological dimension in noetherian rings, Proc.\ Nat.\ Acad.\ Sci.\ U.S.A.\ \textbf{42} (1956).
\bibitem[AB2]{AB2} \bysame Homological dimension in local rings, Trans.\ Amer.\ Math.\ Soc.\ \textbf{85} (1957) 390-405. 
\bibitem[BKM]{BKM} K.\ Baur, A.\ King, R.\ Marsh, Dimer models and cluster categories of Grassmannians, Proc.\ London Math.\ Soc.\ (2016). 
\bibitem[B1]{B1} C.\ Beil, Cyclic contractions of dimer algebras always exist, arXiv:1703.04450.
\bibitem[B2]{B2} \bysame, Homotopy dimer algebras and cyclic contractions, in preparation.
\bibitem[B3]{B3} \bysame, Morita equivalences and Azumaya loci from Higgsing dimer algebras, J.\ Algebra \textbf{453} (2016) 429-455.
\bibitem[B4]{B4} \bysame, Noetherian criteria for dimer algebras, in preparation. 
\bibitem[B5]{B5} \bysame, Nonnoetherian geometry, J.\ Algebra Appl.\ \textbf{15} (2016).
\bibitem[B6]{B6} \bysame, On the central geometry of nonnoetherian dimer algebras, in preparation.
\bibitem[B7]{B7} \bysame, On the noncommutative geometry of square superpotential algebras, J.\ Algebra \textbf{371} (2012) 207-249.
\bibitem[B8]{B8} \bysame, The central nilradical of nonnoetherian dimer algebras, in preparation.
\bibitem[BFHMS]{BFHMS} S.\ Benvenuti, S.\ Franco, A.\ Hanany, D.\ Martelli, J.\ Sparks, An infinite family of superconformal quiver gauge theories with Sasaki-Einstein duals, J.\ High Energy Phys.\ \textbf{06} (2005) 064. 
\bibitem[BD]{BD} D.\ Berenstein and M.\ Douglas, Seiberg duality for quiver gauge theories, (2002) arXiv:hep-th/0207027. 
\bibitem[Bo]{Bo} R.\ Bocklandt, Consistency conditions for dimer models, Glasgow Math.\ J.\ \textbf{54} (2012) 429-447.
\bibitem[Br]{Br} N.\ Broomhead, Dimer models and Calabi-Yau algebras, Memoirs AMS (2012) 1011. 
\bibitem[BH]{BH} K.\ Brown and C.\ Hajarnavis, Homologically homogeneous rings, Trans.\ Amer.\ Math.\ Soc.\ \textbf{281} (1984) 197-208. 
\bibitem[D]{D} B.\ Davison, Consistency conditions for brane tilings, J.\ Algebra \textbf{338} (2011) 1-23.
\bibitem[C]{C} P.\ L.\ Clark, Commutative Algebra, unpublished.  Available at:\\ http://math.uga.edu/~pete/MATH8020.html.
\bibitem[G]{G} D.\ Gulotta, Properly ordered dimers, R-charges, and an efficient inverse algorithm, J.\ High Energy Phys.\ (2008).
\bibitem[IU]{IU} A.\ Ishii and K.\ Ueda, Dimer models and the special McKay correspondence, Geometry and Topology \textbf{19} (2015) 3405-3466.
\bibitem[R]{R} J.\ Rotman, \textit{An introduction to homological algebra}, Springer, 2009.
\bibitem[S]{S} J.-P.\ Serre, Sur la dimension homologique des anneaux et des modules noeth\'eriens. (French) Proceedings of the international symposium on algebraic number theory, Tokyo and Nikko (1955) 175-189. Science Council of Japan, Tokyo (1956).
\bibitem[V]{V} M.\ Van den Bergh, Non-commutative crepant resolutions, The legacy of Niels Henrik Abel, Springer, Berlin (2004) 749-770.
\end{thebibliography}
\def\cprime{$'$} \def\cprime{$'$}

\end{document}